\documentclass[11pt]{amsart}

 \usepackage{amsfonts,graphics,amsmath,amsthm,amsfonts,amscd,amssymb,amsmath,latexsym,multicol,
 mathrsfs}
\usepackage{epsfig,url}
\usepackage{flafter}
\usepackage{fancyhdr}
\usepackage{hyperref}
\hypersetup{colorlinks=true, linkcolor=black}

 \usepackage[matrix, arrow]{xy}

\DeclareMathOperator{\Div}{Div}

\DeclareMathOperator{\lct}{lct}

\DeclareMathOperator{\Pic}{Pic}

\DeclareMathOperator{\Supp}{Supp}

\DeclareMathOperator{\vol}{vol}

 \numberwithin{equation}{subsection}
 \numberwithin{footnote}{subsection}

 \newtheorem{lem}[subsection]{Lemma}
 \newtheorem{prop}[subsection]{Proposition}
 \newtheorem{thm}[subsection]{Theorem}
 \newtheorem{conj}[subsection]{Conjecture}

{
\theoremstyle{upright}

 \newtheorem{defn}[subsection]{Definition}
 \newtheorem{exa}[subsection]{Example}
 
 \newtheorem{rem}[subsection]{Remark}

}

 \newcommand{\N}{\mathbb N}
 \newcommand{\PP}{\mathbb P}
 
 \newcommand{\Q}{\mathbb Q}
 \newcommand{\R}{\mathbb R}
 \newcommand{\Z}{\mathbb Z}
  \newcommand{\C}{\mathbb C}
 \newcommand{\bir}{\dashrightarrow}
 \newcommand{\rddown}[1]{\left\lfloor{#1}\right\rfloor} 

\title{\large B\MakeLowercase{oundedness of} F\MakeLowercase{ano type fibrations}}
\author{\large C\MakeLowercase{aucher} B\MakeLowercase{irkar}}
\date{\today}
\begin{document}
\maketitle

\begin{abstract}
In this paper, we prove various results on boundedness and singularities of Fano fibrations and of Fano type fibrations. A Fano fibration is a projective morphism $X\to Z$ of algebraic varieties with connected fibres 
such that $X$ is Fano over $Z$, that is, $X$ has ``good" singularities and $-K_X$ is ample over $Z$. 
A Fano type fibration is similarly defined where $X$ is assumed to be close to being Fano over $Z$.
This class includes many central ingredients of birational geometry such as Fano varieties, Mori fibre spaces, 
flipping and divisorial contractions, crepant models, germs of singularities, etc. We develop the theory in 
the more general framework of log Calabi-Yau fibrations.\\ 

Dans cet article, nous prouvons divers r\'esultats sur les limites et les singularit\'es de
fibrations de Fano et les fibrations de type Fano. Une fibration de Fano est un morphisme projectif $X\to Z$ de vari\'et\'es alg\'ebriques \`a fibres connexes tel que $X$ est Fano sur $Z$, c'est-\`a-dire que $X$ a de  ``bonnes" singularit\'es et $-K_X$ est ample sur $Z$. Une fibration de type Fano est d\'efinie de façon similaire quand $X$ est suppos\'e \^etre proche d'\^etre Fano sur $Z$. Cette classe comprend de nombreux ingr\'edients centraux de g\'eom\'etrie birationnelle tels que les vari\'et\'es de fano, les espaces de fibres Mori, le flip et les contractions divisorielles, les mod\`eles r\'ep\'etiteurs, les germes de singularit\'es, etc. Nous d\'eveloppons la th\'eorie dans le cadre plus g\'en\'eral des log-fibrations de Calabi-Yau.

\end{abstract}

\setcounter{tocdepth}{1} \tableofcontents


\section{\bf Introduction}

We work over a fixed algebraically closed field $k$ of characteristic zero unless stated otherwise. 
Varieties are assumed to be irreducible.

According to the minimal model  program (including the abundance conjecture),  
every variety $W$ is expected to be birational to a projective variety $X$ 
with good singularities such that either 
\begin{itemize}
\item $X$ is canonically polarised (i.e. $K_X$ is ample), or

\item  $X$ admits a Mori-Fano fibration $X\to Z$ (i.e. $K_X$ is anti-ample over $Z$), or 

\item $X$ admits a Calabi-Yau fibration $X\to Z$ (i.e. $K_X$ is numerically trivial over $Z$).
\end{itemize} 
This reduces the birational classification 
of algebraic varieties to classifying such  $X$. From the point of view of moduli theory,  
it makes perfect sense to focus on such $X$, as they have a better chance of having a 
reasonable moduli theory, due to the special geometric structures they carry. For this 
and other reasons, Fano and Calabi-Yau varieties and their fibrations are central to birational geometry. 
They are also of great importance in many other parts of mathematics such as arithmetic geometry, differential 
geometry, mirror symmetry, and mathematical physics.

Boundedness properties of canonically polarised varieties and Fano varieties have been extensively studied  
in the literature leading to recent advances \cite{HMX-moduli, B-compl, B-BAB}. 
With the above philosophy of the minimal model program in mind, 
there is a natural urge to extend such studies to Fano and Calabi-Yau fibrations. 
Such fibrations also frequently appear in inductive arguments. 

In this paper, a \emph{Fano fibration} means a projective morphism $X\to Z$ of algebraic varieties with connected fibres 
such that $X$ is Fano over $Z$, that is, $X$ has klt singularities and $-K_X$ is ample over $Z$ (for definition of klt and lc singularities, see \ref{ss-pairs}). 
The following are general guiding questions which are the focus of this paper:

\medskip

{\it (1) Under what conditions do Fano fibrations form bounded families?} 

\smallskip

{\it (2) How do singularities behave on the total space and base of a Fano fibration?}  
\smallskip

{\it (3) When do bounded (klt or lc) complements exist for Fano fibrations?}

\medskip

More precise formulations of these questions include many important and hard problems. 
When $Z$ is just a point, these questions were studied in \cite{B-BAB, B-compl}. 
In this paper we are mostly interested in the case $\dim Z>0$ which poses new challenges 
that do not appear in the case $\dim Z=0$.

One of our main results concerns Question (1). Recall that the BAB conjecture states that 
Fano varieties of given dimension $d$ and with $\varepsilon$-lc singularities form a bounded family, 
where $\varepsilon>0$. We would like to extend this result to the relative setting. 
Consider Fano fibrations $f\colon X\to Z$ where $Z$ 
is projective and varies in a bounded family. What kind of conditions should be imposed to ensure that 
the $X$ form a bounded family? Let us assume that $X$ has dimension $d$ and has $\varepsilon$-lc singularities, 
with $\varepsilon>0$. It turns out that this is not enough even when $Z=\PP^1$ and $X$ is a smooth 
surface, see Example \ref{exa-ruled-surfaces}. To get boundedness, we need more subtle 
conditions. A special case of our main results says that if we take a very ample divisor 
$A$ on $Z$ with bounded $A^{\dim Z}$ so that $f^*A-K_X$ is nef, then we can ensure $X$ varies 
in a bounded family. For more general statements see Theorems \ref{t-bnd-cy-fib}, \ref{t-log-bnd-cy-fib}, 
\ref{t-log-bnd-cy-gen-fib}.

In order to prove Theorems \ref{t-bnd-cy-fib} and \ref{t-log-bnd-cy-fib}, 
we will need to treat variants of Questions (2) and (3).  
Indeed, we find bounded natural numbers $m,n$ and produce a boundary divisor $\Lambda$ so that 
$n(K_X+\Lambda)\sim mf^*A$ (see Theorem \ref{t-bnd-comp-lc-global-cy-fib}). 
This is a kind of bounded klt complement as in Question (3), so 
we are led deep into the theory of complements. To construct this complement, in turn we need 
to control singularities which is related to Question (2) (see 
Theorems \ref{t-sing-FT-fib-totalspace}, \ref{t-sh-conj-bnd-base}). 

In our definition of Fano fibrations, we also allow $f$ to be birational. 
In this case, we get boundedness of crepant models which appears frequently in applications.  
This is explained in more details below.

The results of this paper are motivated by the classification theory of algebraic varieties.  
Theorem \ref{t-bnd-cy-fib} can be considered as the first step towards the classification 
of Fano fibrations (see the beginning of Section 2). On the other hand, such a boundedness statement 
is often required to perfom inductive proofs and to deduce boundedness properties from 
birational models. In fact, Theorem \ref{t-bnd-cy-fib} and other results of this paper 
have found many important applications. Here we recall a partial list:
\begin{itemize}
\item boundedness of polarised varieties \cite{B-pol-var},
\item boundedness and volumes of generalised pairs which are in turn 
applied to questions on varieties of intermediate Kodaira dimension and stable minimal models 
\cite{B-bnd-gen-pairs}, 
\item boundedness of elliptic Calabi-Yau varieties \cite{BDS} (also see \cite{DiCerbo-Svaldi}),
\item boundedness of rationally connected 3-folds $X$ with $-K_X$ nef \cite{BDS, CDHJS}.
\end{itemize}

In the rest of this introduction, we state our main results. We will use  
the language of log Calabi-Yau fibrations. In the next section, we state more general results 
in the context of generalised pairs. 
     
\bigskip

{\textbf{\sffamily{Log Calabi-Yau and Fano type fibrations.}}}
Now we introduce the notion which unifies many central ingredients of birational geometry. 
A \emph{log Calabi-Yau fibration} consists of a pair $(X,B)$ with log canonical singularities and 
a contraction $f\colon X\to Z$ (i.e. a surjective projective morphism 
with connected fibres) such that $K_X+B\sim_\R 0$ relatively over $Z$. 
We usually denote the fibration by $(X,B)\to Z$. 
Note that we allow the two extreme cases: when $f$ is birational and when $f$ is constant.
When $f$ is birational, such a fibration is a crepant model of $(Z,f_*B)$ (see below).
When $f$ is constant, that is, when $Z$ is a point, we just say $(X,B)$ is a \emph{log Calabi-Yau pair}.  
In general, if $F$ is a general fibre of $f$ and if we let $K_F+B_F=(K_X+B)|_F$, then $K_F+B_F\sim_\R 0$, 
hence $(F,B_F)$ is a log Calabi-Yau pair justifying the terminology. 

The class of log Calabi-Yau fibrations includes all log Fano and log Calabi-Yau varieties 
 and much more. For example, if $X$ is a variety which is Fano over a base $Z$, then 
we can easily find $B$ so that $(X,B)\to Z$ is a log Calabi-Yau fibration. This includes all Mori fibre spaces. 
Since we allow birational contractions, it also includes all divisorial and flipping contractions. 
Another interesting example of log Calabi-Yau fibrations $(X,B)\to Z$ is when  
$X\to Z$ is the identity morphism; the set of such fibrations simply coincides with the set of 
pairs with log canonical singularities.
On the other hand, a surface with a minimal elliptic fibration over a curve is another instance of 
a log Calabi-Yau fibration.

A log Calabi-Yau fibration $(X,B)\to Z$ is of \emph{Fano type} if $X$ is of Fano type over $Z$, 
that is, if $-(K_X+C)$ is ample over $Z$ and $(X,C)$ is klt for some boundary $C$. When 
$(X,B)$ is klt, this is equivalent to saying that $-K_X$ is big over $Z$.

 We introduce some notation, somewhat similar to \cite{Jiang}, 
to simplify the statements of our results below. 

\begin{defn}\label{d-FT-fib}
Let $d,r$ be natural numbers and $\varepsilon$ be a positive real number. 
A \emph{$(d,r,\varepsilon)$-Fano type (log Calabi-Yau) fibration} consists of a 
pair $(X,B)$ and a contraction $f\colon X\to Z$ such that we have the following:
\begin{itemize}
\item $(X,B)$ is a projective $\varepsilon$-lc pair of dimension $d$,

\item $K_X+B\sim_\R f^* L$ for some $\R$-divisor $L$, 

\item $-K_X$ is big over  $Z$, i.e. $X$ is of Fano type over $Z$,

\item $A$ is a very ample divisor on $Z$ with $A^{\dim Z}\le r$, and 

\item $A-L$ is ample.
\end{itemize}
\end{defn}

That is, a $(d,r,\varepsilon)$-Fano type fibration is a projective log Calabi-Yau fibration which is of Fano type and with certain 
geometric and numerical data bounded by the numbers $d,r,\varepsilon$. The condition $A^{\dim Z}\le r$ 
means that $Z$ belongs to a bounded family of varieties. Ampleness of $A-L$ means that 
the ``degree" of $K_X+B$ is in some sense bounded (this degree is measured with respect to $A$).
When $Z$ is a point, the last two conditions in the definition are vacuous: in this case the fibration is 
simply a Fano type $\varepsilon$-lc  Calabi-Yau pair of dimension $d$.

We are now ready to state some of the main results of this paper. To keep the introduction 
as simple as possible, we have moved further results and remarks to Section 2.\\

{\textbf{\sffamily{Boundedness of Fano type fibrations.}}}
Our first result concerns the boundedness of Fano type fibrations as defined above. 
This is a relative version 
of the so-called BAB conjecture \cite[Theorem 1.1]{B-BAB} which is about boundedness of Fano varieties 
in the global setting. 

\begin{thm}\label{t-bnd-cy-fib}
Let $d,r$ be natural numbers and $\varepsilon$ be a positive real number. Consider the set of all
$(d,r,\varepsilon)$-Fano type fibrations $(X,B)\to Z$ as in \ref{d-FT-fib}. 
Then the $X$ form a bounded family.
\end{thm}

The theorem also holds in the 
more general setting of generalised pairs, see \ref{t-log-bnd-cy-gen-fib}.  

Jiang \cite[Theorem 1.4]{Jiang} considers a setting similar to that of the theorem and proves birational boundedness 
of $X$ modulo several conjectures. We use some of his arguments to get the birational boundedness 
 but we need to do a lot more work to get boundedness.

The boundedness statement of Theorem \ref{t-bnd-cy-fib} does not say anything about boundedness of 
$\Supp B$. This is because in general we have no control over $\Supp B$, e.g.  
when $X=\PP^2$ and $Z$ is a point, $\Supp B$ can contain arbitrary 
hypersufaces. However, if the coefficients of $B$ are bounded away from zero, then indeed 
$\Supp B$ would also be bounded. More generally we have: 
 
\begin{thm}\label{t-log-bnd-cy-fib}
Let $d,r$ be natural numbers and $\varepsilon,\delta$ be positive real numbers. Consider the set of all
$(d,r,\varepsilon)$-Fano type fibrations $(X,B)\to Z$ as in \ref{d-FT-fib} and $\R$-divisors   
$0\le \Delta\le B$ where the non-zero coefficients of $\Delta$ are $\ge \delta$.
Then the set of such $(X,\Delta)$ is log bounded.
\end{thm}
\ 

{\textbf{\sffamily{Boundedness of crepant models.}}}
It is interesting to look at the special cases of Theorem \ref{t-bnd-cy-fib}  when $f$ is birational 
and when it is constant. In the latter case, the theorem is equivalent to the 
BAB conjecture \cite[Theorem 1.1 and Corollary 1.4]{B-BAB},  but 
in the former case, which says something about \emph{crepant models}, a lot work is needed to 
derive it from the BAB. 

Given a pair $(Z,B_Z)$ and a birational contraction $\phi \colon X\to Z$, we can write  
$$
K_X+B=\phi^*(K_Z+B_Z)
$$ 
for some uniquely determined $B$. We say $(X,B)$ is a 
crepant model of $(Z,B_Z)$ if $B\ge 0$. The birational case of Theorem \ref{t-bnd-cy-fib} 
then essentially says that if $Z$ belongs to a bounded family, if the ``degree" of $B_Z$ is bounded 
with respect to some very ample divisor, and if $(Z,B_Z)$ is $\varepsilon$-lc, then   
the underlying varieties of all the crepant models of such pairs form a bounded family; 
this is quite non-trivial even in the case $Z=\PP^3$ (actually it is already challenging for $Z=\PP^2$ 
if we do not use BAB). 
Special cases of 
boundedness of crepant models have appeared in the literature assuming that $\Supp B_Z$ is bounded, see 
\cite[Lemma 10.5]{Mc-Pr}, \cite[Propositions 2.5, 2.9]{HX}, \cite[Proposition 4.8]{CDHJS}. The key 
point here is that we remove such assumptions on the support of $B_Z$.  

Note that the $\varepsilon$-lc condition and 
boundedness of  ``degree" of $B_Z$ are both necessary. Indeed if we replace $\varepsilon$-lc by 
lc, then the crepant models will not be bounded, e.g. 
considering $(Z=\PP^2,B_Z)$ where $B_Z$ is the union of three lines intersecting transversally and   
successively blowing up intersection points in the boundary gives an infinite sequence of 
crepant models with no bound on their Picard numbers. On the other hand, if 
$B_Z$ can have arbitrary degree, then we can easily choose it so that  
$(Z=\PP^2,B_Z)$ is $\frac{1}{2}$-lc having a crepant model of arbitrarily large Picard number.\\

{\textbf{\sffamily{Singularities on Fano type fibrations.}}}
Understanding singularities on log Calabi-Yau fibrations 
is very important as it naturally appears in inductive arguments. 
The next statement gives a lower bound for lc thresholds on Fano type fibrations. 
In particular, it implies \cite[Conjecture 1.13]{Jiang} as a special case. 

\begin{thm}\label{t-sing-FT-fib-totalspace}
Let $d,r$ be natural numbers and $\varepsilon$ be a positive real number. 
 Then there is a positive real number $t$ depending only on $d,r,\varepsilon$ satisfying the following. 
Let $(X,B)\to Z$ be any $(d,r,\varepsilon)$-Fano type fibration as in \ref{d-FT-fib}. If 
$P\ge 0$ is any $\R$-Cartier divisor on $X$ such that either  
\begin{itemize}
\item $f^*A+B-P$ is pseudo-effective, or 

\item $f^*A-K_X-P$ is pseudo-effective, 
\end{itemize}
then  
$(X,B+t P)$ is klt.
\end{thm}

In particular, the theorem can be applied to any 
$$
0\le P\sim_\R f^*A+B
$$ 
or any 
$$
0\le P\sim_\R f^*A-K_X
$$ 
assuming $P$ is $\R$-Cartier. To get a feeling for what the theorem says consider the case 
when $Z$ is a curve; in this case the theorem implies that the multiplicities of each fibre of $f$ 
are bounded from above (compare with the main result of \cite{Mori-Prokhorov-del-pezzo} 
for del Pezzo fibrations over curves): 
indeed, for any closed point $z\in Z$ we can find $0\le Q\sim A$ so that 
$z$ is a component of $Q$; then applying the theorem to $P:=f^*Q$ implies that the multiplicities of 
the fibre of $f$ over $z$ are bounded.

One can derive the theorem from \ref{t-log-bnd-cy-fib} and the results of \cite{B-BAB}. 
However, in practice the theorem is proved together along with \ref{t-log-bnd-cy-fib} in an intertwining inductive process.

On the other hand, a fundamental problem on singularities is a conjecture of 
Shokurov \cite[Conjecture 1.2]{B-sing-fano-fib} (a special case of which is due to M$^{\rm c}$Kernan) 
which roughly says that the singularities on 
the base of a Fano type fibration are controlled by those on the total space.
We will prove this conjecture under some boundedness assumptions on the 
base (\ref{t-sh-conj-bnd-base}). 
This result is of independent interest but also closely related to the other results of this paper. 

First we recall adjuction for fibrations, also known as canonical bundle formula. If $(X,B)$ is an lc pair
and $f\colon X\to Z$ is a contraction with $K_X+B\sim_\R 0/Z$, then 
by \cite{kaw-subadjuntion, ambro-adj} we can define a 
discriminant divisor $B_Z$ and a moduli divisor $M_Z$ so that we have 
$$
K_X+B\sim_\R f^*(K_Z+B_Z+M_Z).
$$
This is a generalisation of the Kodaira canonical bundle formula. 
Let $D$ be a prime divisor on $Z$. Let $t$ be the lc threshold of $f^*D$ with respect to $(X,B)$ 
over the generic point of $D$. 
We then put the coefficient of  $D$ in $B_Z$ to be $1-t$. 
Having defined $B_Z$, we can find $M_Z$ giving 
$$
K_{X}+B\sim_\R f^*(K_Z+B_Z+M_Z)
$$
where $M_Z$ is determined up to $\R$-linear equivalence. 
We call $B_Z$ the \emph{discriminant divisor of adjunction} for $(X,B)$ over $Z$. 

For any birational contraction $Z'\to Z$, from a normal variety there is 
a birational contraction $X'\to X$ from a normal variety so that the induced map $X'\bir Z'$ is a morphism.  
Let $K_{X'}+B'$ be the 
pullback of $K_{X}+B$. We can similarly define $B_{Z'},M_{Z'}$ for $(X',B')$ over $Z'$. In this way we 
get the \emph{discriminant b-divisor ${\bf{B}}_Z$ of adjunction} for $(X,B)$ over $Z$. 

The conjecture of Shokurov then can be stated as:

\begin{conj}\label{conj-sh-sing-fib}
Let $d$ be a natural number and $\varepsilon$ be a positive real number. 
 Then there is a positive real number $\delta$ depending only on $d,\varepsilon$ satisfying the following. 
 Assume that $(X,B)$ is a pair and $f\colon X\to Z$ is a contraction such that
\begin{itemize}
\item $(X,B)$ is $\varepsilon$-lc of dimension $d$,

\item $K_X+B\sim_\R 0/Z$, and  

\item $-K_X$ is big over  $Z$.
\end{itemize}
Then the discriminant b-divisor ${\bf B}_Z$ has coefficients in $(-\infty, 1-\delta]$.
\end{conj}

The next result says that Shokurov conjecture holds in the setting of $(d,r,\varepsilon)$-Fano type fibrations.
The strength of this result is in the fact that (similar to some of the other results above, e.g. \ref{t-bnd-cy-fib}) we are 
not assuming any boundedness condition on support of $B$ along the general fibres of $f$. 
In contrast the main result of \cite{B-sing-fano-fib} proves Shokurov conjecture when the general fibres of $f$ together with the support of $B$ restricted to these fibres belong to a bounded family. The strength of 
\cite{B-sing-fano-fib} is in the fact that the base is not assumed projective and so there is no global condition on the numerical class of $K_X+B$.

\begin{thm}\label{t-sh-conj-bnd-base}
Let $d,r$ be natural numbers and $\varepsilon$ be a positive real number. 
 Then there is a positive real number $\delta$ depending only on $d,r,\varepsilon$ satisfying the following. 
Let $(X,B)\to Z$ be any  $(d,r,\varepsilon)$-Fano type fibration as in \ref{d-FT-fib}. 
Then  the discriminant b-divisor ${\bf B}_Z$ has coefficients in $(-\infty, 1-\delta]$.
\end{thm}

This can be viewed as a relative version of Ambro's conjecture \cite[Theorem 1.6]{B-BAB} which is closely related 
to the BAB conjecture.\\

{\textbf{\sffamily{Boundedness of klt relative-global complements.}}}
One of the key tools used in this paper is the theory of complements. 
The complements we consider are relative but are controlled globally. Moreover, they have klt singularities. 
In general, klt complements are much harder to construct than lc complements.

\begin{thm}\label{t-bnd-comp-lc-global-cy-fib}
Let $d,r$ be natural numbers, $\varepsilon$ be a positive real number,  and 
$\mathfrak{R}\subset [0,1]$ be a finite set of rational numbers. Then there exist natural numbers 
$n,m$ depending only on $d,r,\varepsilon,\mathfrak{R}$ satisfying the following. 
Assume that $(X,B)\to Z$ is a $(d,r,\varepsilon)$-Fano type fibration and that 
\begin{itemize}
\item  we have  $0\le \Delta\le B$ and the coefficients of $\Delta$ are in $\mathfrak{R}$, and 

\item $-(K_X+\Delta)$ is big over $Z$. 
\end{itemize}
Then there is a $\Q$-divisor $\Lambda\ge \Delta$ such that  
\begin{itemize}
\item $(X,\Lambda)$ is klt, and 

\item $n(K_X+\Lambda)\sim mf^*A$.
\end{itemize}
\end{thm}

For example, we can apply the theorem under the assumptions of \ref{t-bnd-cy-fib} by taking 
$\Delta=0$.\\


{\textbf{\sffamily{Acknowledgements.}}
This work was mainly done at Cambridge University with support of a grant of the Leverhulme Trust. 
Part of it was done while visiting the 
University of Tokyo in March 2018 which was arranged by Yujiro Kawamata and Yoshinori Gongyo, 
and I would like to thank them for their hospitality. It was revised at Tsinghua University.
Thanks to Christopher Hacon for answering our 
questions regarding moduli part of adjunction.  Thanks to Yifei Chen for numerous valuable comments.
Thanks to Jingjun Han, Roberto Svaldi, and Yanning Xu for their useful comments on an earlier version of this paper, and thanks to the referee for helpful comments and suggestions.

\bigskip

\section{\bf Further results and remarks}

In this section, we state further results and remarks.\\

{\textbf{\sffamily{A framework for classification of Fano fibrations.}}
Here we illustrate how the results above can be used towards classification of 
Fano fibrations such as Mori fibre spaces  
in the context of birational classification of algebraic varieties.
Suppose that we are given a normal projective variety $Z$ with a very ample divisor $A$ on it. The aim is 
to somehow classify a given set of Fano fibrations over $Z$. We naturally want to fix or bound 
certain invariants. Let $d$ be a natural number and $\varepsilon<1$ be a positive real number. 
Assume $\mathcal{P}$ is a set of contractions $f\colon X\to Z$ such that 
\begin{itemize}
\item $X$ is projective of dimension $d$ with $\varepsilon$-lc singularities, and 

\item $-K_X$ is ample over $Z$. 
\end{itemize}
For each non-negative integer $l$, let $\mathcal{P}_l$ be the set of all $X \to Z$ in $\mathcal{P}$ 
such that 
$$
l=\min\{a\in \Z^{\ge 0} \mid \mbox{$af^*A-K_X$ is ample}\}.
$$ 
For example, $X\to Z\in \mathcal{P}_0$ means that $-K_X$ is ample, hence $X$ is globally a Fano variety; 
 $X\to Z\in \mathcal{P}_1$ means that $-K_X$ is not ample but $f^*A-K_X$ is ample. Thus we have a disjoint union 
$$
\mathcal{P}=\bigcup_{l\in \Z^{\ge 0}} \mathcal{P}_l.
$$
For each $X\to Z$ in $\mathcal{P}_l$, we can choose a general 
$$
0\le B\sim_\Q lf^*A-K_X
$$ 
so that  
$(X,B)\to Z$ is a 
$$
\mbox{$(d,((l+1)A)^{\dim Z},{\varepsilon})$-Fano type fibration}. 
$$
Thus the set of such  
$X$ forms a bounded family, by Theorem \ref{t-bnd-cy-fib}. That is, we can write $\mathcal{P}$ as a disjoint union of 
bounded sets. The next step is to study each set $\mathcal{P}_l$ more closely to get a finer classification.

\begin{exa}\label{exa-ruled-surfaces}\emph{
Let us look at the simplest non-trivial case of surfaces, that is, consider the set $\mathcal{P}$ of 
Mori fibre spaces $X\to Z=\PP^1$ where $d=\dim X=2$ and $X$ is smooth.  In this case it is well-known that 
$\mathcal{P}$ coincides with the sequence of Hirzebruch surfaces, that is, $\PP^1$-bundles $f_i\colon X_i\to Z$ having a section 
$E_i$ satisfying $E_i^2=-i$, for $i=0,1,\dots$. Applying the divisorial 
adjunction formula gives $K_{X_i}\cdot E_i=i-2$. 
Letting $A$ be a point on $Z$ and using the fact that the Picard group of $X_i$ is generated by $E_i$ and 
a fibre of $f_i$, it is easy to check that $X_0$ and $X_1$ are Fano, and $(i-1)f_i^*A-K_{X_i}$ is ample
but $(i-2)f_i^*A-K_{X_i}$ is not ample for $i\ge 2$. Therefore,  under the notation introduced above, we have 
$$
\mathcal{P}_0=\{X_0\to Z,X_1\to Z\}, ~~\mbox{and}~~ \mathcal{P}_l=\{X_{l+1}\to Z\}~~ \mbox{for}~~ l\ge 1.
$$
Of course we have used the classification of ruled surfaces over $\PP^1$ but the point we want to make is that 
conversely studying $\mathcal{P}$ and each subset $\mathcal{P}_l$ closely will naturally lead us  
to the classification of ruled surfaces over $\PP^1$.}\
\end{exa}

{\textbf{\sffamily{Boundedness of generalised Fano type fibrations.}}}
We will prove the results stated above in the context of generalised pairs. This is important for further 
applications of these results. 
For the basics of generalised pairs, see \cite{BZh} and \cite{B-compl} and the preliminaries below. 
A \emph{generalised log Calabi-Yau fibration} consists of a generalised pair $(X,B+M)$ 
with generalised lc singularities and a contraction $X\to Z$ such that 
$$
K_X+B+M\sim_\R 0/Z.
$$ 
We define generalised Fano type fibrations  similar to \ref{d-FT-fib}. 

\begin{defn}\label{d-gen-FT-fib}
Let $d,r$ be natural numbers and $\varepsilon$ be a positive real number. 
A generalised $(d,r,\varepsilon)$-Fano type (log Calabi-Yau) fibration consists of a projective generalised pair 
$(X,B+M)$ with data $X'\to X$ and $M'$ and a contraction $f\colon X\to Z$ such that we have: 
\begin{itemize}
\item $(X,B+M)$ is generalised $\varepsilon$-lc of dimension $d$,

\item $K_X+B+M\sim_\R f^* L$ for some $\R$-divisor $L$,

\item $-K_X$ is big over  $Z$, i.e. $X$ is of Fano type over $Z$,

\item $A$ is a very ample divisor on $Z$ with $A^{\dim Z}\le r$, and 

\item $A-L$ is ample.   
\end{itemize}
\end{defn}
We usually write $(X,B+M)\to Z$ to denote the fibration.
Theorem \ref{t-log-bnd-cy-fib} can be extended to the case of generalised pairs, that is:

\begin{thm}\label{t-log-bnd-cy-gen-fib}
Let $d,r$ be natural numbers and $\varepsilon,\tau$ be positive real numbers. Consider the set of all
generalised $(d,r,\varepsilon)$-Fano type fibrations $(X,B+M)\to Z$ such that we have  $0\le \Delta\le B$ and the non-zero coefficients of $\Delta$ are $\ge \tau$.
Then the set of such $(X,\Delta)$ is log bounded.
\end{thm}

Note that $M'$ is assumed to be nef globally in \ref{d-gen-FT-fib}. If we only assume that $M'$ is nef over $Z$, then the theorem does not hold. Indeed, if $X\to Z=\PP^1$ is a ruled surface as in Example \ref{exa-ruled-surfaces} 
and if we let $X'=X$ and $M'=-K_X$, then $K_X+M=0$ and all the assumptions of the theorem are satisfied (except global nefness of $M'$) with $(d,r,\varepsilon)=(2,1,1)$ and with $\Delta=0$ but such $X$ are not bounded.\\ 

{\textbf{\sffamily{Singularities on generalised Fano type fibrations.}}}
Theorem \ref{t-sing-FT-fib-totalspace} also holds for generalised pairs, that is: 

\begin{thm}\label{t-sing-gen-FT-fib-totalspace}
Let $d,r$ be natural numbers and $\varepsilon$ be a positive real number. 
 Then there is a positive real number $t$ depending only on $d,r,\varepsilon$ satisfying the following. 
Let $(X,B+M)\to Z$ be any generalised $(d,r,\varepsilon)$-Fano type fibration as in \ref{d-gen-FT-fib}. If 
$P\ge 0$ is any $\R$-Cartier divisor on $X$ such that either  
\begin{itemize}
\item $f^*A+B+M-P$ is pseudo-effective, or 

\item $f^*A-K_X-P$ is pseudo-effective, 
\end{itemize}
then  
$(X,B+t P+M)$ is generalised klt.
\end{thm}

In particular, the theorem can be applied to any 
$$
0\le P\sim_\R f^*A+B+M
$$ 
or any 
$$
0\le P\sim_\R f^*A-K_X
$$ 
assuming $P$ is $\R$-Cartier. 

Adjunction for fibrations also makes sense for generalised pairs. That is, if 
$$
(X,B+M)\to Z
$$ 
is a 
generalised log Calabi-Yau fibration, then we can define a discriminant divisor $B_Z$ and a moduli divisor 
$M_Z$ giving 
$$
K_X+B+M\sim_\R f^*(K_Z+B_Z+M_Z).
$$
Moreover, for any birational contraction $Z'\to Z$ from a normal variety we can define the discriminant divisor $B_{Z'}$ 
whose pushdown to $Z$ is just $B_Z$. In this way we get the discriminant b-divisor  ${\bf B}_Z$.
See \ref{fib-adj-setup} for more details.

Now we state a generalised version of Shokurov's conjecture \ref{conj-sh-sing-fib}.

\begin{conj}\label{conj-sh-sing-gen-fib}
Let $d$ be a natural number and $\varepsilon$ be a positive real number. 
 Then there is a positive real number $\delta$ depending only on $d,\varepsilon$ satisfying the following. 
 Assume that $(X,B+M)$ is a generalised pair with data $X'\to X\overset{f}\to Z$ and $M'$ 
where $f$ is a contraction such that
\begin{itemize}
\item $(X,B+M)$ is generalised $\varepsilon$-lc of dimension $d$,

\item $K_X+B+M\sim_\R 0/Z$, and  

\item $-K_X$ is big over $Z$.
\end{itemize}
Then the discriminant b-divisor ${\bf B}_Z$ has coefficients in $(-\infty, 1-\delta]$.
\end{conj}

Note that, in particular, we are assuming that $M'$ is nef over $Z$ as this is part of the definition of a generalised pair. 
The next result says that \ref{conj-sh-sing-gen-fib} holds in the setting of generalised Fano type fibrations.

\begin{thm}\label{t-sh-conj-bnd-base-gen-fib}
Let $d,r$ be natural numbers and $\varepsilon$ be a positive real number. 
 Then there is a positive real number $\delta$ depending only on $d,r,\varepsilon$ satisfying the following. 
Let $(X,B+M)\to Z$ be any generalised $(d,r,\varepsilon)$-Fano type fibration as in \ref{d-gen-FT-fib}. 
Then  the discriminant b-divisor ${\bf B}_Z$ has coefficients in $(-\infty, 1-\delta]$.
\end{thm}\

{\textbf{\sffamily{Plan of the paper.}}}
We will prove  Theorems \ref{t-bnd-cy-fib}, \ref{t-log-bnd-cy-fib}, 
\ref{t-log-bnd-cy-gen-fib}, \ref{t-sing-FT-fib-totalspace},  \ref{t-bnd-comp-lc-global-cy-fib} in Section 4, and 
Theorems \ref{t-sh-conj-bnd-base}, \ref{t-sing-gen-FT-fib-totalspace}, and
\ref{t-sh-conj-bnd-base-gen-fib} in Section 5.

\bigskip

\section{\bf Preliminaries}

All the varieties in this paper are quasi-projective over a fixed algebraically closed field $k$ of characteristic zero
unless stated otherwise. 

\subsection{Numbers}

Let $\mathfrak{R}$ be a subset of $[0,1]$. Following [\ref{PSh-II}, 3.2] we define 
$$
\Phi(\mathfrak{R})=\left\{1-\frac{r}{m} \mid r\in \mathfrak{R},~ m\in \N\right\}
$$
to be the set of \emph{hyperstandard multiplicities} associated to $\mathfrak{R}$.

\subsection{Contractions}

By a \emph{contraction} we mean a projective morphism $f\colon X\to Y$ of varieties 
such that $f_*\mathcal{O}_X=\mathcal{O}_Y$ ($f$ is not necessarily birational). In particular, 
$f$ is surjective and has connected fibres.

\subsection{Divisors}\label{ss-divisors}

Let $X$ be a variety. If $D$ is a prime divisor on birational models of $X$ whose centre on $X$ is non-empty, 
then we say $D$ is a prime divisor \emph{over} $X$. 

Now let $X$ be normal and $M$ be an $\R$-divisor on $X$. The coefficient of a prime divisor $D$ in $M$ is denoted $\mu_DM$. More generally, if $M$ is $\R$-Cartier and if $D$ is a prime divisor over $X$, $\mu_DM$ means the coefficient of $D$ in the pullback of $M$ to any resolution of $X$ on which $D$ is a divisor.

Again let $X$ be normal and $M$ be an $\R$-divisor on $X$.
We let 
$$
|M|_\R=\{N \ge 0\mid N\sim_\R M\}.
$$
Recall that $N\sim_\R M$ means that 
$N-M=\sum r_i\Div(\alpha_i)$ for certain real numbers $r_i$ and rational functions $\alpha_i$. 
When all the $r_i$ can be chosen to be rational numbers, then we write $N\sim_\Q M$.
We define $|M|_\Q$ similarly by replacing $\sim_\R$ with $\sim_\Q$. 

Assume $\rho\colon X\bir Y/Z$ is a rational map of normal varieties projective over a base variety $Z$. 
For an $\R$-Cartier divisor $L$ on $Y$ we define the pullback $\rho^*L$ as follows. Take a 
common resolution $\phi\colon W\to X$ and $\psi\colon W\to Y$. Then let $\rho^*L:=\phi_*\psi^*L$. 
It is easy to see that this does not depend on the choice of the common resolution as 
any two such resolutions are dominated by a third one.

\begin{lem}\label{l-semi-ample-div}
Assume $Y\to X$ is a contraction of normal projective varieties, $C$ is a nef $\R$-divisor on 
$Y$ and $A$ is the pullback of an ample $\R$-divisor on $X$. If $C$ is semi-ample over $X$, then 
$C+aA$ is semi-ample (globally) for any real number $a>0$.
\end{lem}
\begin{proof}
Since $C$ is semi-ample over $X$, it defines a contraction $\phi \colon Y\to Z/X$ to a normal projective variety. Replacing $Y$ 
with $Z$ and replacing $C,A$ with $\phi_*C,\phi_*A$, respectively, we can assume $C$ is ample over $X$. Pick $a>0$. 
Now $C+bA$ is ample for some $b\gg a$ because 
$A$ is the pullback of an ample divisor on $X$. Since $C$ is globally nef, 
$$
C+tbA=(1-t)C+t(C+bA)
$$ 
is ample for any $t\in (0,1]$. In particular, taking $t=\frac{a}{b}$ we see that $C+aA$ is ample.
\end{proof}

\subsection{Linear systems}

Let $X$ be a normal projective variety and $M$ be an integral Weil divisor on $X$. 
A \emph{sub-linear system} $L\subseteq |M|$ is given by some linear subspace $V\subseteq \PP(H^0(M))$, that is, 
$$
L=\{\Div(\alpha)+M \mid \alpha \in V\}.
$$
The general members of $L$, by definition, are those $\Div(\alpha)+M$ 
where $\alpha$ is in some given non-empty open subset $W\subseteq V$ (open in the Zariski topology). 
Being a general member then depends on the choice of $W$ but we usually shrink it if necessary without notice.

\begin{lem}\label{l-gen-element-lin-system}
Let $X$ be a normal projective variety, $M$ be an integral Weil divisor on $X$, and $L\subseteq |M|$ be a sub-linear system. 
Assume that $x\in X$ is a smooth closed point and that some member of $L$ is smooth at $x$. 
Then a general member of $L$ is smooth at $x$.
\end{lem}
\begin{proof}
Since $M$ is Cartier at $x$, we can move $M$ hence assume $x$ is not contained in $\Supp M$.
Then considering $H^0(M)$ as a subset of the function field $K$ of $X$, each element of $H^0(M)$ is regular at $x$. Now there is a linear subspace $V'\subset H^0(M)$ so that $\PP(V')=V$. 
There is a $k$-linear map $\theta\colon V'\to \mathcal{O}_{X,x}$ to the local ring at $x$. Let $m_x$ be the maximal ideal of $\mathcal{O}_{X,x}$. Then $\theta^{-1}m_x^2$ is a linear subspace of $V'$.
Moreover, the divisor given by an element $0\neq \alpha\in V'$ is smooth at $x$ iff multiplicity of the divisor at $x$ is $\le 1$ iff $\alpha \notin \theta^{-1}m_x^2$ (note that multiplicity $0$ is also possible because the divisor may not pass through $x$). 
 
By assumption, there is $D\in  L$ such that $D$ is smooth at $x$. So there is $0\neq \alpha\in V'$ whose divisor is smooth at $x$. Thus  $V'\neq \theta^{-1}m_x^2$, so the general members of $V$ give divisors that are smooth at $x$ because $\PP(\theta^{-1}m_x^2)\subsetneq V$. 
\end{proof}

\begin{lem}\label{l-gen-element-v.ample-div-passing-x,y}
Let $X$ be a normal projective variety and $A$ be a very ample divisor on $X$. 
For each pair of closed points $x,y\in X$, let $L_{x,y}$ be the sub-linear system of $|2A|$ 
consisting of the members that pass through both $x,y$. Then there is a 
non-empty open subset $U\subseteq X$ such that for any pair of closed points $x,y\in U$, a general 
member of $L_{x,y}$ is smooth at both $x,y$.
\end{lem}
\begin{proof}
A general member of $|A|$ is the element given by a section in some non-empty open subset 
$W$ of $\PP(H^0(A))$. Perhaps after shrinking $W$, we can assume that the restriction of these 
general members to the smooth locus of $X$ are smooth. 

We claim that there is a finite set $\Pi$  of closed points of $X$ (depending on $W$) such 
that for each closed point $x\in X\setminus \Pi$ 
 we can find a general member of $|A|$ passing through $x$:  indeed since $A$ is very ample, 
the set $H_z$ of elements of $\PP(H^0(A))$ vanishing at a given closed point $z$ is a hyperplane, 
and for distinct points $z,z'$ we have $H_z\neq H_{z'}$; but there are at most finitely many $z$ 
with $W\cap H_z=\emptyset$ because the complement of $W$ in $\PP(H^0(A))$ is a proper closed set; 
hence for any closed point $x$ other than those finite set we can 
find an element of $W$ vanishing at $x$ which proves the claim.  
Thus there is a non-empty open subset $U$ of the smooth locus of $X$ 
such that  for each closed point $x\in U$ we can find a 
member of $|A|$ passing through $x$ which is smooth at $x$.

Now pick a closed point $x\in U$ and let $L_x$ be the sub-linear system of $|A|$ consisting of 
members passing through $x$. By the above arguments some member of $L_x$ is smooth at $x$, 
hence a general member of 
$L_x$ is also smooth at $x$, by Lemma \ref{l-gen-element-lin-system}. 
In particular, for any other closed point $y\in U$, we can pick a member of  
$L_x$ smooth at $x$ but not containing $y$.

Now let $x,y\in U$ be a pair of distinct closed points and let $L_{x,y}$ be the sub-linear system 
of $|2A|$ consisting of members passing through $x,y$. 
By the previous paragraph, there exists a member $D$ (resp. $E$) of $|A|$ which passes through and 
smooth at $x$ (resp. $y$) but not containing $y$ (resp. $x$). Then $D+E$ is a member of $L_{x,y}$ 
passing through $x,y$ and smooth at both $x,y$. Therefore, a general member  of $L_{x,y}$ 
is smooth at both $x,y$, by Lemma \ref{l-gen-element-lin-system}.
\end{proof}

\subsection{Pairs and singularities}\label{ss-pairs}
A \emph{pair} $(X,B)$ consists of a normal variety $X$ and 
an $\R$-divisor $B\ge 0$ such that $K_X+B$ is $\R$-Cartier. 
Let $\phi\colon W\to X$ be a log resolution of $(X,B)$ and let 
$$
K_W+B_W=\phi^*(K_X+B).
$$
 The \emph{log discrepancy} of a prime divisor $D$ on $W$ with respect to $(X,B)$ 
is $1-\mu_DB_W$ and it is denoted by $a(D,X,B)$.
We say $(X,B)$ is \emph{lc} (resp. \emph{klt})(resp. \emph{$\varepsilon$-lc}) 
if $a(D,X,B)$ is $\ge 0$ (resp. $>0$)(resp. $\ge \varepsilon$) for every $D$. Note that if $(X,B)$ is $\varepsilon$-lc, then 
automatically $\varepsilon\le 1$ because $a(D,X,B)\le 1$ if $D$ is a prime divisor on $X$. 

A \emph{non-klt place} of $(X,B)$ is a prime divisor $D$ over $X$, that is, 
on birational models of $X$, such that $a(D,X,B)\le 0$. A \emph{non-klt centre} is the image on 
$X$ of a non-klt place. 

\emph{Sub-pairs} and their singularities are defined similarly by letting the coefficients of $B$ to be any real number. 
In this case instead of lc, klt, etc, we say sub-lc, sub-klt, etc.

\begin{lem}\label{l-gen-element-v.ample-div-passing-x,y-contraction}
Let $(X,B)$ be a projective $\varepsilon$-lc pair for some $\varepsilon>0$ and $f\colon X\to Z$ be a contraction. 
Let $A$ be a very ample divisor on $Z$. 
For each pair of closed points $z,z'\in Z$, let $L_{z,z'}$ be the sub-linear system of $|2A|$ 
consisting of the members that pass through $z,z'$. Then there is a 
non-empty open subset $U\subseteq Z$ such that if $z,z'\in U$ are closed points and  if 
$H$ is a general member of $L_{z,z'}$, then $(X,B+G)$ is a plt pair where 
$G:=f^*H$. In particular, $G$ is normal and $(G,B_G)$ is an $\varepsilon$-lc pair where 
$$
K_G+B_G:=(K_X+B+G)|_G.
$$
\end{lem}
\begin{proof}
Let $U$ be as in Lemma \ref{l-gen-element-v.ample-div-passing-x,y} chosen for $|2A|$ on $Z$.
We can assume that $U$ is contained in the smooth locus of $Z$. 
Let $\phi\colon W\to X$ be a log resolution of $(X,B)$ and let $\Sigma$ be the union of the exceptional divisors of 
$\phi$ and the birational transform of $\Supp B$. Shrinking $U$ we can assume that for any stratum $S$ 
of $(W,\Sigma)$ the morphism $S\to Z$ is smooth over $U$.  A stratum of $(W,\Sigma)$ is either  
$W$ itself or an irreducible component of 
$\bigcap_{i\in I} D_i$ for some $I\subseteq \{1,\dots,r\}$ where $D_1,\dots,D_r$ are the irreducible components of 
$\Sigma$. For each stratum $S$ we can assume that either 
$S\to Z$ is surjective or that its image is contained in $Z\setminus U$. 

Pick closed points $z,z'\in U$ and a general member $H$ of $L_{z,z'}$, and let $G=f^*H$ and $E=\phi^*G$. 
We claim that $(W,\Sigma+E)$ is log smooth. 
Let $S$ be a stratum of $(W,\Sigma)$. 
Since $L_{z,z'}$ is base point 
free outside $z,z'$ by definition of $L_{z,z'}$, 
the pullback of $L_{z,z'}$ to $S$ is base point free outside the fibres of $S\to Z$ over $z,z'$. Thus 
any singular point of $E|_S$ (if there is any) is mapped to  $z$ or $z'$. In particular, if 
$S\to Z$ is not surjective, then $E|_S$ is smooth because in this case $E|_S$ has no point 
mapping to either $z$ or $z'$. Assume that $S\to Z$ is surjective. 
Then $E|_S\to H$ is surjective. Moreover, 
by our choice of $U$, the fibres of $E|_S\to H$ over $z,z'$ are both smooth as they coincide with the 
fibres of $S\to Z$ over $z,z'$. Therefore, $E|_S$ is 
smooth because $H$ is smooth at $z,z'$ by the choice of $U$ and by 
Lemma \ref{l-gen-element-v.ample-div-passing-x,y}. To summarise we have shown that 
$E|_S$ is smooth for each stratum $S$ of $(W,\Sigma)$.  This implies that $(W,\Sigma+E)$ is log smooth. 

Let $K_W+B_W$ be the pullback of $K_X+B$. Then each coefficient of $B_W$ is $\le 1-\varepsilon$. 
Since $K_W+B_W+E$ is the pullback of $K_X+B+G$, we deduce that $(X,B+G)$ is plt, hence 
$G$ is normal \cite[Proposition 5.51]{kollar-mori}. On the other hand, 
$$
K_E+B_E:=K_E+B_W|_E=(K_W+B_W+E)|_E
$$ 
is the pullback of 
$$
K_G+B_G:=(K_X+B+G)|_G.
$$
Moreover, since $(W,B_W+E)$ is log smooth and since $E$ is not a component of $B_W$,  
the coefficients of $B_E$ are each at most $1-\varepsilon$. Therefore, $(G,B_G)$ is an 
$\varepsilon$-lc pair.
\end{proof}

\subsection{Fano type varieties}
Assume $X$ is a variety and $X\to Z$ is a contraction. We say $X$ is \emph{of Fano type over} $Z$ 
if there is a boundary $C$ such that $(X,C)$ is klt and $-(K_X+C)$ is ample over $Z$ 
(or equivalently if there is a boundary $C$ such that $(X,C)$ is klt and $-(K_X+C)$ is nef and big over $Z$). This is equivalent to 
having a boundary $B$ such that $(X,B)$ is klt, $K_X+B\sim_\R 0/Z$ and $B$ is big over $Z$. 
By \cite{BCHM}, we can run an MMP over $Z$ on any $\R$-Cartier $\R$-divisor $M$ on $X$ and the MMP ends 
with a minimal model or a Mori fibre space for $M$.

\begin{lem}\label{l-div-pef-on-FT}
Assume $f\colon X\to Z$ is a contraction of normal projective varieties, $X$ is of Fano type over $Z$, $L$ is an $\R$-divisor on 
$X$ and $A$ is an ample $\R$-divisor on $Z$. If $L$ is pseudo-effective, then 
$|L+af^*A|_\R\neq \emptyset$  for any real number $a>0$. If $L$ is pseudo-effective and $L$ is big over $Z$, then $L+af^*A$ is big for any real number $a>0$.
\end{lem}
\begin{proof}
Since  $L$ is pseudo-effective and $X$ is of Fano type over $Z$, $L$ has a minimal model over $Z$. 
Replacing $X$ with the minimal model we can assume $L$ is semi-ample over $Z$. Thus $L$ defines 
a contraction $X\to Y/Z$ and $L$ is the pullback of an ample$/Z$ $\R$-divisor $N$ on $Y$. Then 
$N+bg^*A$ is ample for any $b\gg 0$ where $g$ denotes $Y\to Z$. Since $L$ is pseudo-effective, 
$N$ is pseudo-effective, hence 
$$
|N+tbg^*A|_\R=|(1-t)N+t(N+bg^*A)|_\R \neq \emptyset
$$ 
for any $t\in (0,1]$.
In particular, if $b>a>0$, then letting $t=\frac{a}{b}$ we see that $|N+ag^*A|_\R\neq \emptyset$ which in turn implies 
$|L+af^*A|_\R\neq \emptyset$.

If additionally $L$ is big over $Z$, then $X\to Y$ is birational and similar arguments show that $L+af^*A$ is big for any real number $a>0$.
\end{proof}

\subsection{Complements}\label{ss-compl}
Let $(X,B)$ be a pair  and let $X\to Z$ be a contraction. 
A \emph{strong $n$-complement} of $K_{X}+B$ over a point $z\in Z$ is of the form 
$K_{X}+{B}^+$ such that over some neighbourhood of $z$ we have the following properties:
\begin{itemize}
\item $(X,{B}^+)$ is lc, 

\item $n(K_{X}+{B}^+)\sim 0$, and 

\item ${B}^+\ge B$.
\end{itemize}

When $Z$ is a point, we just say that $K_X+B^+$ is a strong $n$-complement of $K_X+B$. We recall one of the main 
results of \cite{B-compl} and one of the main results of \cite{B-BAB} on complements.

\begin{thm}[{\cite[Theorem 1.7]{B-compl}}]\label{t-bnd-compl-usual}
Let $d$ be a natural number and $\mathfrak{R}\subset [0,1]$ be a finite set of rational numbers.
Then there exists a natural number $n$ 
depending only on $d$ and $\mathfrak{R}$ satisfying the following.  
Assume $(X,B)$ is a projective pair such that 
\begin{itemize}

\item $(X,B)$ is lc of dimension $d$,

\item  the coefficients of $B$ are in $\Phi(\mathfrak{R})$, 

\item $X$ is of Fano type, and 

\item $-(K_{X}+B)$ is nef.\
\end{itemize}
Then there is a strong $n$-complement $K_{X}+{B}^+$ of $K_{X}+{B}$. 
Moreover, the complement is also a strong $mn$-complement for any $m\in \N$.
\end{thm}

\begin{thm}[{\cite[Theorem 1.9]{B-BAB}}]\label{t-bnd-comp-lc-global}
Let $d$ be a natural number and 
$\mathfrak{R}\subset [0,1]$ be a finite set of rational numbers. Then there exists a natural number 
$n$ depending only on $d,\mathfrak{R}$ satisfying the following. Assume 

\begin{itemize}
\item $(X,B)$ is a projective lc pair of dimension $d$,

\item the coefficients of $B$ are in $\mathfrak{R}$, 

\item $M$ is a semi-ample Cartier divisor on $X$ defining a contraction $f\colon X\to Z$,

\item $X$ is of Fano type over $Z$, 

\item $M-(K_X+B)$ is nef and big, and 

\item $S$ is a non-klt centre of $(X,B)$ with $M|_{S}\equiv 0$.\\ 

\end{itemize}
Then there is a $\Q$-divisor $\Lambda\ge B$ such that  
\begin{itemize}
\item $(X,\Lambda)$ is lc over a neighbourhood of $z:=f(S)$, and

\item $n(K_X+\Lambda)\sim (n+2)M$.
\end{itemize}
\end{thm}

\subsection{Bounded families of pairs}\label{ss-bnd-couples}
A \emph{couple} $(X,D)$ consists of a normal projective variety $X$ and a  divisor 
$D$ on $X$ whose non-zero coefficients are all equal to $1$, i.e. $D$ is a reduced divisor. 
The reason we call $(X,D)$ a couple rather than a pair is that we are concerned with 
$D$ rather than $K_X+D$ and we do not want to assume $K_X+D$ to be $\Q$-Cartier 
or with nice singularities. Two couples $(X,D)$ and $(X',D')$ are isomorphic 
(resp. isomorphic in codimension one) if 
there is an isomorphism  $X\to X'$ (resp. birational map $X\bir X'$ which is an isomorphism in codimension one) 
mapping $D$ onto $D'$ (resp. such that $D$ is the birational transform of $D'$).

We say that a set $\mathcal{P}$ of couples  is \emph{birationally bounded} if there exist 
finitely many projective morphisms $V^i\to T^i$ of varieties and reduced divisors $C^i$ on $V^i$ 
such that for each $(X,D)\in \mathcal{P}$ there exist an $i$, a closed point $t\in T^i$, and a 
birational isomorphism $\phi\colon V^i_t\bir X$ such that $(V^i_t,C^i_t)$ is a couple and 
$E\le C_t^i$ where 
$V_t^i$ and $C_t^i$ are the fibres over $t$ of the morphisms $V^i\to T^i$ and $C^i\to T^i$ 
respectively, and $E$ is the sum of the 
birational transform of $D$ and the reduced exceptional divisor of $\phi$.
We say $\mathcal{P}$ is \emph{bounded} if we can choose $\phi$ to be an isomorphism. 

We say that a set $\mathcal{P}$ of couples  is \emph{bounded up to isomorphism in codimension one} 
if there is a bounded set $\mathcal{P}'$ of couples such that each $(X,D)\in \mathcal{P}$ is 
isomorphic in codimension one with some $(X',D')\in \mathcal{P}'$.
  
A set $\mathcal{R}$ of projective pairs $(X,B)$ is said to be \emph{log birationally bounded} 
(resp. \emph{log bounded}, etc) 
if the set of the corresponding couples $(X,\Supp B)$ is birationally bounded (resp. bounded, etc).
Note that this does not put any condition on the coefficients of $B$, e.g. we are not requiring the 
coefficients of $B$ to be in a finite set. If $B=0$ for all the $(X,B)\in\mathcal{R}$ we usually remove the 
log and just say the set is birationally bounded (resp. bounded, etc).

\begin{lem}\label{l-bnd-Cartier-index-family}
Let $\mathcal{P}$ be a bounded set of couples and $\mathfrak{R}\subset [0,1]$ be a finite set of 
rational numbers. Then there is a natural number $I$ depending only on $\mathcal{P},\mathfrak{R}$ 
such that if 
\begin{itemize}
\item $(X,B)$ is a projective klt pair,

\item the coefficients of $B$ are in $\mathfrak{R}$, and 

\item $(X,\Supp B)\in \mathcal{P}$,
\end{itemize}
then $I(K_X+B)$ is Cartier.
\end{lem}

\begin{proof}
Assume there is a sequence $(X_i,B_i)$ of pairs  as in the lemma such that if $I_i$ 
 is the smallest natural number so that $I_i(K_{X_i}+B_i)$ is Cartier, 
 then the $I_i$ form a strictly increasing sequence of numbers.  Perhaps after replacing 
 the sequence with a subsequence 
we can assume there is a projective morphism $V\to T$ of varieties, a reduced divisor $C$ on 
$V$, and a dense set of closed points $t_i\in T$ such that $X_i$ is the fibre of $V\to T$ over 
$t_i$ and  $\Supp B_i$ is contained in the fibre of $C\to T$ over $t_i$. Since $X_i$ are normal, 
replacing $V$ with its normalisation 
and replacing $C$ with its inverse image with reduced structure, we can assume $V$ is normal.

Let $\phi\colon W\to V$ be a log resolution of $(V,C)$ and let $\Sigma$ be the 
union of the birational transform of $C$ and the reduced exceptional divisor of $\phi$. After a finite base change and possibly shrinking $T$ we can assume that $T$ is smooth, $(W,\Sigma)$ is relatively log smooth over $T$ and $S\to T$ has irreducible fibres for each stratum of $(W,\Sigma)$.
Let $E$ be the reduced exceptional divisor of $W\to V$. 
Let $W_i$ be the fibre of $W\to T$ over $t_i$.
We can assume that $E|_{W_i}$ is the reduced exceptional divisor of $W_i\to X_i$. 

Since $(X_1,B_1)$ is klt, it is $\varepsilon$-lc for some $\varepsilon>0$. Let $\Delta_1$ be the birational transform of $B_1$ plus $(1-\frac{\varepsilon}{2})E|_{W_1}$. Then $\Supp \Delta_1\subseteq \Supp \Sigma|_{W_1}$. Since the fibres of $S\to T$ are irreducible for each component $S$ of $\Sigma$, there is a unique divisor $\Delta$ supported in $\Sigma$ so that $\Delta|_{W_1}=\Delta_1$. 
Note that $(W,\Delta)$ is klt and the coefficients of $\Delta$ belong to $\mathfrak{R}\cup\{1-\frac{\varepsilon}{2}\}$. 
Possibly replacing the sequence from the beginning, we can assume that the pushdown of $\Delta|_{W_i}$ to $X_i$ is equal to $B_i$ for every $i$. 

Running an MMP on $K_W+\Delta$ over $V$ ends with a minimal model $(W',\Delta')$. We claim that $W'\to V$ is a small birational morphism. Let $W_1'$ be the fibre of $W'\to T$ over $t_1$. Possibly shrinking $T$ around $t_1$ and using arguments similar to those of the proof of Lemma \ref{l-gen-element-v.ample-div-passing-x,y-contraction} we can choose smooth irreducible divisors $H_1, \dots,H_{\dim T}$ on $T$ so that if $S_i$ is the pullback of $H_i$ to $W$, then $W_1=\bigcap S_j$ and 
$$
(W,\Lambda:=\sum S_j+\Delta)
$$ 
is log smooth. Then the MMP is also an MMP on $K_W+\Lambda$. From this, we can see that $W_1'$ is just the birational transform of $W_1$ because any component of $W_1'$ is a non-klt centre of $(W',\Lambda')$ and any such centre is the birational transform of a non-klt centre of $(W,\Lambda)$, and because $W_1$ is the only non-klt centre of $(W,\Lambda)$ mapping to $t_1$. This shows that $W_1'\to X_1$ is birational.

By construction and by adjunction we have 
$$
K_{W_1}+\Delta_1=(K_{W}+\Lambda)|_{W_1}.
$$ 
Again by adjunction we can write 
$$
K_{W_1'}+\Lambda_1'=(K_{W'}+\Lambda')|_{W_1'}.
$$ 
Let $\alpha \colon U\to W_1$ and $\beta\colon U\to W_1'$ be a common resolution. 
Then since $W\bir W'$ is given by an MMP on $K_W+\Lambda$, we have 
$$
\alpha^*(K_{W_1}+\Delta_1)\ge \beta^*(K_{W_1'}+\Lambda_1').
$$
Thus since the pushdown of $K_{W_1}+\Delta_1$ to $X_1$ is $K_{X_1}+B_1$, we see that the pushdown of $K_{W_1'}+\Lambda_1'$ to $X_1$ is $\le K_{X_1}+B_1$. Therefore, since $K_{W_1'}+\Lambda_1'$ is nef over $X_1$, by the negativity lemma, we have 
$$
K_{W_1'}+\Lambda_1'\le \rho^*(K_{X_1}+B_1)
$$
where $\rho$ denotes $W_1'\to X_1$. 
In particular, every coefficient of $\Lambda_1'$ is $\le 1-\varepsilon$ as $(X_1,B_1)$ is $\varepsilon$-lc. This implies that 
no exceptional$/V$ component of $\Lambda'$ intersects $W_1'$, by Lemma \ref{l-coeff-dlt-adj} below, because the coefficient of each exceptional component of $\Lambda'$ is $1-\frac{\varepsilon}{2}$. But this means that every exceptional divisor is contracted by the MMP, hence $W'\to V$ is a small birational contraction. 

Replacing the sequence, we can assume that $W_i'\to X_i$ is a small birational contraction for any $i>1$. Moreover, for such $i$ the pushdown of 
$$
K_{W_i'}+\Lambda_i':=(K_{W'}+\Lambda')|_{W_i'}
$$ 
to $X_i$ is $K_{X_i}+B_i$ because the pushdown of $\Delta|_{W_i}=\Lambda|_{W_i}$ to $X_i$ is $B_i$. 
Since the Cartier index of $K_{W_i'}+\Lambda_i'$ is bounded, the Cartier index of $K_{X_i}+B_i$ is also bounded by the cone theorem, a contradiction.
\end{proof}

\begin{lem}\label{l-coeff-dlt-adj}
Let $(X,B)$ be a dlt pair, let $S$ be a non-klt centre, and write $K_S+B_S=(K_X+B)|_S$ by adjunction. Assume $D\ge 0$ is a $\Q$-Cartier divisor on $X$ such that $S\not\subset \Supp D$ and assume each component of $D$ has coefficient $\ge b$ in $B$. Then the coefficient of each component of $D|_S$ in $B_S$ is $\ge b$.
\end{lem}
\begin{proof}
First assume $S$ is of codimension one, i.e. a component of $\rddown{B}$. Let $P$ be a component of $D|_S$. Then some component of $B$ with coefficient $\ge b$ contains $P$.
So the coefficient of $P$ in $B_S$ is of the form $1-\frac{1}{n}+\frac{c}{n}$ for some $n\in \N$ and real number $c\ge b$ \cite[Corollary 3.10]{Shokurov-log-flips}. We can check easily that this coefficient is at least $b$. If $S$ has higher codimension, then we pick a component $T$ of $\rddown{B}$ containing $S$ and prove the lemma inductively by replacing $(X,B)$ with $(T,B_T)$ and replacing $D$ with $D|_T$ where $K_T+B_T=(K_X+B)|_T$.  
\end{proof}

\subsection{BAB and lower bound on lc thresholds}

We recall some of the main results of \cite{B-BAB} regarding boundedness of Fano's and lc thresholds in families.

\begin{thm}[{\cite[Theorem 1.1]{B-BAB}}]\label{t-BAB}
Let $d$ be a natural number and $\varepsilon$ a positive real number. Then the projective 
varieties $X$ such that  

$\bullet$ $(X,B)$ is $\varepsilon$-lc of dimension $d$ for some boundary $B$, and 

$\bullet$  $-(K_X+B)$ is nef and big,\\\\
form a bounded family. 
\end{thm}

On the other hand, lc thresholds are bounded from below under suitable assumptions:

\begin{thm}[{{\cite[Theorem 1.8]{B-BAB}}}]\label{t-bnd-lct}
Let $d,r$ be natural numbers and $\varepsilon$ be a positive real number. 
Then  there is a positive real number $t$ depending only on $d,r,\varepsilon$ satisfying the following. 
Assume 

$\bullet$ $(X,B)$ is a projective $\varepsilon$-lc pair of dimension $d$, 

$\bullet$ $A$ is a very ample divisor on $X$ with $A^d\le r$,

$\bullet$ $A-B$ is pseudo-effective, and 

$\bullet$ $M\ge 0$ is an $\R$-Cartier $\R$-divisor with $A-M$ pseudo-effective.\\
Then  
$$
\lct(X,B,|M|_\R)\ge \lct(X,B,|A|_\R)\ge t.
$$
\end{thm}

Note that the conditions on $A,B,M$ essentially say that $X$ belongs to a bounded family and that the 
``degrees" of $B,M$ with respect $A$ are bounded.

\subsection{b-divisors}\label{ss-b-divisor}

We recall some definitions regarding b-divisors but not in full generality. 
Let $X$ be a variety. A \emph{b-$\R$-Cartier b-divisor over $X$} is the choice of  
a projective birational morphism 
$Y\to X$ from a normal variety and an $\R$-Cartier divisor $M$ on $Y$ up to the following equivalence: 
 another projective birational morphism $Y'\to X$ from a normal variety and an $\R$-Cartier divisor
$M'$ defines the same b-$\R$-Cartier  b-divisor if there is a common resolution $W\to Y$ and $W\to Y'$ 
on which the pullbacks of $M$ and $M'$ coincide.  

A b-$\R$-Cartier  b-divisor  represented by some $Y\to X$ and $M$ is \emph{b-Cartier} if  $M$ is 
b-Cartier, i.e. its pullback to some resolution is Cartier.

\subsection{Generalised pairs}

A \emph{generalised pair} consists of 
\begin{itemize}
\item a normal variety $X$ equipped with a projective
morphism $X\to Z$, 

\item an $\R$-divisor $B\ge 0$ on $X$, and 

\item a b-$\R$-Cartier  b-divisor over $X$ represented 
by some projective birational morphism $X' \overset{\phi}\to X$ and $\R$-Cartier divisor
$M'$ on $X'$
\end{itemize}
such that $M'$ is nef$/Z$ and $K_{X}+B+M$ is $\R$-Cartier,
where $M := \phi_*M'$. 

We usually refer to the pair by saying $(X,B+M)$ is a  generalised pair with 
data $X'\to X\to Z$ and $M'$, and call $M'$ the nef part. Note that our notation 
here, which seems to be preferred by others, is slightly different from those in \cite{BZh}\cite{B-compl}. 

We now define generalised singularities.
Replacing $X'$ we can assume $\phi$ is a log resolution of $(X,B)$. We can write 
$$
K_{X'}+B'+M'=\phi^*(K_{X}+B+M)
$$
for some uniquely determined $B'$. For a prime divisor $D$ on $X'$ the \emph{generalised log discrepancy} 
$a(D,X,B+M)$ is defined to be $1-\mu_DB'$. 
We say $(X,B+M)$ is 
\emph{generalised lc} (resp. \emph{generalised klt})(resp. \emph{generalised $\varepsilon$-lc}) 
if for each $D$ the generalised log discrepancy $a(D,X,B+M)$ is $\ge 0$ (resp. $>0$)(resp. $\ge \varepsilon$).

A \emph{generalised non-klt centre} of a generalised pair $(X,B+M)$ is the image on $X$ of a prime divisor 
$D$ over $X$  with $a(D,X,B+M)\le 0$, and 
the \emph{generalised non-klt locus} of the generalised pair is the union of all the generalised non-klt centres.  

\emph{Generalised sub-pairs} and their singularities are similarly defined by allowing the coefficients of 
$B$ to be any real number.

The next lemma is useful for reducing problems about generalised log Calabi-Yau fibrations to 
usual log Calabi-Yau fibrations.

\begin{lem}\label{l-from-gen-fib-to-usual-fib}
Let $d,r$ be natural numbers and $\varepsilon$ be a positive real number. Let 
$(X,B+M)\to Z$ be a generalised $(d,r,\varepsilon)$-Fano type fibration (as in \ref{d-gen-FT-fib}) such that  
 $-(K_X+\Delta)$ is big over $Z$ for some $0\le \Delta\le B$. 
Then we can find a boundary $\Theta\ge \Delta$ such that $(X,\Theta)\to Z$ is a $(d,r,\frac{\varepsilon}{2})$-Fano type fibration.
\end{lem}
\begin{proof}
Taking a $\Q$-factorialisation we can assume $X$ is $\Q$-factorial.  
Since $B-\Delta$ is effective and $M$ is pseudo-effective (as it is the pushdown of a nef divisor),
$B-\Delta+M$ is pseudo-effective. Moreover, since  $-(K_X+\Delta)$ is big over $Z$, 
$$
B-\Delta+M\sim_\R -(K_X+\Delta)/Z
$$ 
is big over $Z$. Thus  
$$
t(B-\Delta+M)+f^*A
$$ 
is globally big for any sufficiently small $t>0$, hence 
$$
B-\Delta+M+f^*A
$$ 
is globally big as  $B-\Delta+M$ is pseudo-effective.

Let $\phi\colon X'\to X$ be a log resolution of $(X,B)$ 
on which the nef part $M'$ of $(X,B+M)$ resides. Write 
$$
K_{X'}+B'+M'=\phi^*(K_X+B+M)
$$
and 
$$
K_{X'}+\Delta'=\phi^*(K_X+\Delta).
$$
Then
$$
B'-\Delta'+M'=\phi^*(B-\Delta+M).
$$ 
Since $(X,B+M)$ is generalised $\varepsilon$-lc, the coefficients of $B'$ do not exceed $1-\varepsilon$. 
In addition, since $M'$ is nef and $B\ge \Delta$, we have $B'\ge \Delta'$ by the negativity lemma 
applied to $B'-\Delta'$ and the morphism $\phi$.

Since 
$$
B-\Delta+M+f^*A
$$ 
is big, we can write 
$$
B'-\Delta'+M'+\phi^*f^*A=\phi^*(B-\Delta+M+f^*A)\sim_\R G'+H'
$$ 
where $G'\ge 0$ and $H'$ is ample. Replacing $\phi$ we can assume $\phi$ is a log resolution of 
$(X,B+\phi_*G')$.
Pick a small real number $\alpha>0$ and pick a general 
$$
0\le R'\sim_\R \alpha H'+(1-\alpha)M'.
$$
Let 
$$
\Theta':=\Delta'+(1-\alpha)(B'-\Delta')+\alpha G'+R'.
$$
We can make the above choices so that $(X',\Theta')$ is log smooth. Since 
$$
\Delta'\le \Delta'+(1-\alpha)(B'-\Delta')\le \Delta'+(B'-\Delta')=B',
$$ 
we have 
$$
\Delta'\le\Theta'\le B'+\alpha G'+R'.
$$
In particular, we can choose $\alpha$ and $R'$ so that the coefficients of $\Theta'$ do not exceed $1-\frac{\varepsilon}{2}$. 

By construction, we have 
$$
K_{X'}+\Theta'=K_{X'}+ \Delta'+(1-\alpha)(B'-\Delta')+\alpha G'+R'
$$
$$
\sim_\R K_{X'}+  \Delta'+(1-\alpha)(B'-\Delta')+\alpha G'+\alpha H'+(1-\alpha)M'
$$
$$
\sim_\R K_{X'}+ \Delta'+(1-\alpha)(B'-\Delta')+\alpha(B'-\Delta'+M'+ \phi^* f^*A)+(1-\alpha)M'
$$
$$
= K_{X'}+ \Delta'+B'-\Delta'+\alpha(M'+ \phi^* f^*A)+(1-\alpha)M'
$$
$$
=K_{X'}+ B'+M'+\alpha\phi^*f^*A\sim_\R \phi^*f^*(L+\alpha A).
$$
Therefore,   letting $\Theta=\phi_*\Theta'$, we have 
$$
K_X+\Theta\sim_\R f^*(L+\alpha A).
$$ 
Choosing $\alpha$ small enough we can ensure $A-(L+\alpha A)$ is 
ample. Moreover, since $K_{X'}+\Theta' \sim_\R 0/X$ and since the coefficients of $\Theta'$ do not 
exceed $1-\frac{\varepsilon}{2}$, the pair $(X,\Theta)$ is $\frac{\varepsilon}{2}$-lc.
Thus $(X,\Theta)\to Z$ is a $(d,r,\frac{\varepsilon}{2})$-Fano type fibration. Finally, it is obvious that $\Theta\ge \Delta$.
\end{proof}

\subsection{Mori fibre spaces}

The next lemma is similar to results in \cite{DiCerbo-Svaldi}. 

\begin{lem}\label{l-contraction-after-flops}
Suppose that 
\begin{itemize}
\item $(X,B)$ is a $\Q$-factorial projective klt pair where $B$ is a $\Q$-divisor, 

\item $h\colon X\to Y$ is a Mori fibre space structure, i.e. a non-birational $K_X$-negative extremal  contraction,

\item  $Y\to Z$ is a contraction,  

\item $K_X+B\sim_\Q 0/Z$, and

\item $Y\bir Y'/Z$ is a birational map to a $\Q$-factorial normal projective variety 
which is an isomorphism in codimension one.
\end{itemize}
Then there exists a birational map $X\bir X'$ to a $\Q$-factorial normal projective variety which is an isomorphism 
in codimension one and such that the induced map $X'\bir Y'$ is an extremal contraction (hence a morphism).
\end{lem}
\begin{proof}
Since $K_X+B \sim_\Q 0/Y$, by adjunction, we can write 
$$
K_X+B\sim_\Q h^*(K_Y+B_Y+M_Y)
$$ 
where we consider $(Y,B_Y+M_Y)$ as a generalised pair (see \ref{rem-base-fib-gen-pair} below) 
which is generalised klt. Since $X$ is $\Q$-factorial and since $X\to Y$ is extremal and 
non-birational, $Y$ is also $\Q$-factorial \cite[Corollary 3.18]{kollar-mori}.  

Since 
$$
K_Y+B_Y+M_Y\sim_\Q 0/Z,
$$ 
by the same arguments as in \cite{Kaw-flops} we can decompose $Y\bir Y'$ into a sequence of flops: 
indeed if $H_{Y'}$ is an ample $\Q$-divisor on $Y'$, then after rescaling $H_{Y'}$, 
$$
(Y,B_Y+H_Y+M_Y)
$$ 
is 
generalised klt where $H_Y$ is the birational transform of $H_{Y'}$; now running an MMP on 
$$
K_Y+B_Y+H_Y+M_Y
$$ 
over $Z$ ends with $Y'$ as 
$$
K_{Y'}+B_{Y'}+H_{Y'}+M_{Y'}
$$ 
is ample; 
in particular, only flips can occur in the MMP which are flops with respect to $K_Y+B_Y+M_Y$. 

By the previous paragraph, to obtain $X'$ it is enough to consider the case when $Y\bir Y'/Z$ is one single flop, in particular, we can assume $Y\to Z$ is an extremal flopping contraction. 
Let $H_{Y'},H_{Y}$ be as before, and let $G$ be the pullback of $H_Y$ to $X$. 
Since $B$ is big over $Y$ and since $Y\to Z$ is birational, $B$ is also big over $Z$.
Thus $X$ is of Fano type over $Z$ as $(X,B)$ is klt and $K_X+B\sim_\Q 0/Z$. 
Therefore, there is a minimal model $X'$ for $G$ over $Z$. 
The birational transform of $G$ on $X'$, say $G'$, is semi-ample over $Z$, hence it defines a contraction 
$X'\to V'/Z$ and $G'$ is the pullback of an ample$/Z$ divisor $H_{V'}$. By construction, 
$V'$ is the ample model of $G$ over $Z$. Since $G$ is the pullback of $H_Y$, $V'$ is also the ample 
model of $H_Y$ over $Z$. But the ample model of $H_Y$ is $Y'$, hence $V'=Y'$. In particular, 
$X'\bir Y'$ is a morphism. 

Finally, since  $X\to Z$ has relative Picard number two,  $X'\to Z$ also has relative Picard number two 
as $X\bir X'$ is an isomorphism in codimension one. On the other hand, $Y'\to Z$ has relative 
Picard number one, so $X'\to Y'$ is an extremal contraction. 
\end{proof}

The Lemma also holds if $B$ is an $\R$-boundary because we can write $B=\sum r_iB_i$ where 
$r_i\ge 0$ are real numbers with $\sum r_i=1$, $(X,B_i)$ is klt and $K_{X}+B_i\sim_\Q 0/Z$. 

\bigskip


\section{\bf Boundedness of Fano type fibrations}

In this section, we treat boundedness properties of Fano type log Calabi-Yau fibrations. 
We will frequently refer to Definition \ref{d-FT-fib} and use the notation therein.

\subsection{Numerical boundedness}
We start with bounding numerical properties.
The next statement and its proof are similar to \cite[Lemma 3.2]{Jiang}.

\begin{prop}\label{l-bnd-cy-numerical-bndness}
Let $d,r$ be natural numbers and $\varepsilon$ be a positive real number. 
Assume that Theorem \ref{t-sing-FT-fib-totalspace} holds in dimension $d-1$. Then there is a 
natural number $l$ depending only on $d,r,\varepsilon$ satisfying the following. 
Let $(X,B)\to Z$ be a $(d,r,\varepsilon)$-Fano type fibration (as in \ref{d-FT-fib}) such that 
\begin{itemize}
\item $-(K_X+\Delta)$ is nef over $Z$ for some $\R$-divisor $\Delta\ge 0$, and 

\item $f^*A+B-\Delta$ is pseudo-effective. 
\end{itemize}
Then $lf^*A-(K_X+\Delta)$ is nef (globally). 
\end{prop}
\begin{proof}
\emph{Step 1.}
\emph{In this step, we do some basic preparations and introduce some notation.}
Replacing $X$ with a $\Q$-factorialisation, we can assume $X$ is $\Q$-factorial. All 
the assumptions of the proposition are preserved. Put 
$$
C:=f^*A-(K_X+B).
$$
Let $m\ge 2$ be a natural number. We can write 
$$
 \begin{array}{l l}
 (m+2)f^*A-(K_X+\Delta) &  = 2f^*A+m(K_X+B)+mC-(K_X+\Delta) \\
 & = 2f^*A+m(K_X+B)-(K_X+\Delta)+mC \\
 & =2f^*A+(m-1)(K_X+B)+B-\Delta+m C \\
 & =(m-1)(K_X+B+\frac{1}{m-1}(2f^*A+B-\Delta)) +m C.
  \end{array} 
$$

On the other hand, since $f^*A+B-\Delta$ is pseudo-effective and since $X$ is of Fano type over $Z$, there is 
$$
0\le P\sim_\R 2f^*A+B-\Delta
$$
by Lemma \ref{l-div-pef-on-FT}.\\

\emph{Step 2.}
\emph{In this step, we find a real number $t>0$ depending only on $d,r,\varepsilon$ such that 
the non-klt locus of $(X,B+tP)$ is mapped to a finite set of closed points of $Z$.}
We can assume $\dim Z>0$ otherwise the proposition is trivial. 
Let $H\in |A|$ be a general element and let $G=f^*H$, and $g$ be the induced morphism $G\to H$. 
Let
$$
K_G+B_G:=(K_X+B+G)|_G.
$$
By definition of $G$, we have $B_G=B|_G$.
Then 
\begin{itemize}
\item $(G,B_G)$ is $\varepsilon$-lc as $(X,B)$ is $\varepsilon$-lc, 

\item  $-K_G$ is big over $H$ as $-K_G=-(K_X+G)|_G\sim -K_X|_G$ over $H$,

\item we have 
$$
K_G+B_G \sim_\R (f^*L+f^*H)|_G \sim_\R g^*(L+H)|_H\sim_\R g^*(L+A)|_H,
$$ 
and

\item $2A|_H-(L+A)|_H$ is ample as $A-L$ is ample. 
\end{itemize}
Therefore, $(G,B_G)\to H$ is a $(d-1,2^{d-1}r,\varepsilon)$-Fano type fibration.

Letting $P_G=P|_G$ and $Q_G=\Delta|_G$ 
we have 
$$
P_G+Q_G\sim_\R (2f^*A+B-\Delta)|_G+\Delta|_G=(2f^*A+B)|_G\sim g^*2A|_H+B_G.
$$
Since we are assuming Theorem \ref{t-sing-FT-fib-totalspace} in dimension $d-1$, 
we deduce that there is a real number $t>0$ depending only on $d,r,\varepsilon$ such that 
 $(G,B_G+tP_G)$ is klt. Note that here we used the assumption $\Delta\ge 0$ to ensure that 
$g^*2A|_H+B_G-P_G$ is pseudo-effective.

By inversion of adjunction \cite[Theorem 5.50]{kollar-mori} (which is stated for $\Q$-divisors but 
also holds for $\R$-divisors) and by the previous paragraph, 
$$
(X,B+G+tP)
$$ 
is plt near $G$. Since $G$ is a general member of $|f^*A|$, we deduce that the non-klt locus of 
$(X,B+tP)$ (possibly empty) is mapped to a finite set of closed points of $Z$.\\ 

\emph{Step 3.}
\emph{In this step, we consider extremal rays that intersect 
$$
(m+2)f^*A-(K_X+\Delta)
$$ negatively.}
Fix a natural number $m\ge 2$ so that $\frac{1}{m-1}<t$. Let 
$$
\Theta=B+\frac{1}{m-1}P.
$$ 
By Step 1 and by definition of $P$, we have 
$$ 
(m+2)f^*A-(K_X+\Delta)=(m-1)(K_X+B+\frac{1}{m-1}(2f^*A+B-\Delta)) +m C
$$
$$
\sim_\R (m-1)(K_X+B+\frac{1}{m-1}P) +m C=(m-1)(K_X+\Theta)+mC.
$$

Assume that $R$ is an extremal ray of $X$ with 
$$
((m+2)f^*A-(K_X+\Delta))\cdot R<0. 
$$
Since $-(K_X+\Delta)$ is nef over $Z$, $R$ is not vertical over $Z$, that is, $f^*A\cdot R>0$. 
On the other hand, by the previous paragraph,
$$
((m-1)(K_X+\Theta) +m C)\cdot R<0
$$
which in turn implies 
$$
(K_X+\Theta)\cdot R<0
$$
because $C\sim_\R f^*(A-L)$ is nef by Definition \ref{d-FT-fib}.\\
 
\emph{Step 4.}
\emph{In this step, we apply boundedness of length of extremal rays and finish the proof.}
Let $T$ be the non-klt locus of $(X,\Theta)$. By Step 2, $T$ is mapped to a finite set of closed points of $Z$ as $\frac{1}{m-1}<t$. 
Let $V$ be the image of $\overline{NE}(T)\to \overline{NE}(X)$ where by convention we put $\overline{NE}(T)=0$ 
if $T$ is zero-dimensional  or empty. Then $V\cap R=0$ as $f^*A$ intersects 
$R$ positively but intersects every class in $V$ trivially. 
Therefore, by \cite[Theorem 1.1(5)]{Fujino-rays}, 
$R$ is generated by a curve $\Gamma$ with 
$$
-2d\le (K_X+\Theta)\cdot \Gamma,
$$
hence 
$$
-2d(m-1)\le (m-1)(K_X+\Theta)\cdot \Gamma\le ((m-1)(K_X+\Theta) +m C)\cdot \Gamma
$$
$$
=((m+2)f^*A-(K_X+\Delta))\cdot \Gamma.
$$
Moreover, $f^*A\cdot \Gamma \ge 1$.
Therefore, taking 
$$
l=m+2+2d(m-1)
$$ 
we have 
$$
0\le -2d(m-1)+ 2d(m-1)f^*A\cdot \Gamma \le (lf^*A-(K_X+\Delta))\cdot R
$$
ensuring that $lf^*A-(K_X+\Delta)$ is nef.
\end{proof}

\subsection{Bounded very ampleness}
Next we treat very ampleness in a setting similar to the previous proposition.

\begin{lem}\label{l-bnd-cy-fib-v-ampleness}
Let $d,r$ be natural numbers, $\varepsilon$ be a positive real number, and $\mathfrak{R}\subset [0,1]$ be a 
finite set of rational numbers. Assume that Theorem \ref{t-sing-FT-fib-totalspace} holds in dimension $d-1$ 
and that Theorem \ref{t-log-bnd-cy-fib} holds in dimension $d$. 
Then there exist natural numbers $l,m$ depending only on $d,r,\varepsilon,\mathfrak{R}$ satisfying the following. 
If  $(X,B)\to Z$ is a $(d,r,\varepsilon)$-Fano type fibration (as in \ref{d-FT-fib}) such that  
\begin{itemize}
\item we have  $0\le \Delta\le B$ where the coefficients of $\Delta$ are in $\mathfrak{R}$, and

\item $-(K_X+\Delta)$ is ample over $Z$,
\end{itemize} 
then 
$$
m(lf^*A-(K_X+\Delta))
$$ 
is very ample.
\end{lem}

\begin{proof}
Since we are assuming Theorem \ref{t-log-bnd-cy-fib}, $(X,\Delta)$ is log bounded. Thus by 
Lemma \ref{l-bnd-Cartier-index-family}, there is a bounded natural number $m$ such that 
$m(K_X+\Delta)$ is Cartier. On the other hand, by Proposition \ref{l-bnd-cy-numerical-bndness}, 
there is a bounded natural number $l$ such that $lf^*A-(K_X+\Delta)$ is nef. As $-(K_X+\Delta)$ 
is ample over $Z$, replacing $l$ with $2l+1$ we can assume that $lf^*A-2(K_X+\Delta)$ is ample 
(which also implies that $lf^*A-(K_X+\Delta)$ is ample). 

Now  
$$
m(lf^*A-(K_X+\Delta))=K_X+\Delta+(m(lf^*A-(K_X+\Delta))-(K_X+\Delta))
$$ 
where the term 
$$
(m(lf^*A-(K_X+\Delta))-(K_X+\Delta))=m(lf^*A-(1+\frac{1}{m})(K_X+\Delta))
$$
is ample because $1+\frac{1}{m}\le 2$ and $lf^*A-(K_X+\Delta)$ and $lf^*A-2(K_X+\Delta)$ are both ample. 
Thus replacing $m$ we can assume 
$$
|m(lf^*A-(K_X+\Delta))|
$$
is base point free, by the effective base point free theorem \cite[Theorem 1.1]{kollar-ebpf}. Therefore, 
$$
(d+4)m(lf^*A-(K_X+\Delta))
$$
is very ample by \cite[Lemma 1.2]{kollar-ebpf}. Now replace $m$ with $(d+4)m$.
\end{proof}

\subsection{Effective birationality}

The next statement and its proof are somewhat similar to \cite[Lemma 3.3]{Jiang}.

\begin{prop}\label{l-bnd-cy-effective-bir}
Let $d,r$ be natural numbers and $\varepsilon$ be a positive real number.
Assume that Theorems \ref{t-log-bnd-cy-fib} and \ref{t-sing-FT-fib-totalspace} hold in dimension $d-1$. Then there exist 
natural numbers $l,m$ depending only on $d,r,\varepsilon$ satisfying the following. If $(X,B)\to Z$ is a 
$(d,r,\varepsilon)$-Fano type fibration (as in \ref{d-FT-fib}), then the linear system $|m(lf^*A-K_X)|$ defines 
a birational map. 
\end{prop}
\begin{proof}
We first replace $X$ with a $\Q$-factorialisation so that $K_X$ is $\Q$-Cartier. 
Then after running an MMP on $-K_X$ over $Z$,  
we can assume $-K_X$ is nef and big over $Z$. Replacing $X$ with the ample model of $-K_X$ over $Z$, 
we can assume $-K_X$ is ample over $Z$  ($X$ may no longer be $\Q$-factorial but 
we do not need it any more). If $\dim Z=0$, then the proposition holds by \cite[Theorem 1.2]{B-compl}. 
We then assume $\dim Z>0$.
By Proposition \ref{l-bnd-cy-numerical-bndness}, there is $l\in \N$  depending only on $d,r,\varepsilon$  such 
 that $lf^*A-K_X$ is nef.
Since $-K_X$ is ample over $Z$, replacing $l$ with $l+1$ we can assume $lf^*A-K_X$ is ample.

By Lemma \ref{l-gen-element-v.ample-div-passing-x,y-contraction}, there is a 
non-empty open subset $U\subseteq Z$ such that for any pair of closed points $z,z'\in U$ and any general 
member $H$ of the sub-linear system $L_{z,z'}$ of $|2A|$ consisting of elements passing through $z,z'$, 
the pullback $G=f^*H$ is normal and $(G,B_G)$ is an $\varepsilon$-lc pair where 
$$
K_G+B_G=(K_X+B+G)|_G.
$$

Pick distinct closed points $x,x'\in X$ such that $z:=f(x)$ and $z':=f(x')$ are in $U$. 
Let $H,G,B_G$ be as in the previous paragraph constructed for $z,z'$.
Denoting $G\to H$ by $g$ we then have 
\begin{itemize}
\item $(G,B_G)$ is $\varepsilon$-lc, 

\item $K_G+B_G\sim_\R (f^*L+G)|_G\sim g^*(L+2A)|_H$, 

\item $-K_G=-(K_X+G)|_G$ is ample over $H$, 

\item the divisor 
$$
3A|_H-(L+2A)|_H=(A-L)|_H
$$ 
is ample, and 

\item the divisor 
$$
(l+4)g^*A|_H-K_G\sim (l+4)f^*A|_G-(K_X+G)|_G\sim ((l+2)f^*A-K_X)|_G
$$ 
is ample.
\end{itemize}
In particular, $(G,B_G)\to H$ is a $(d-1, 2(3^{d-1})r, \varepsilon)$-Fano type fibration.
Applying Lemma \ref{l-bnd-cy-fib-v-ampleness} and perhaps replacing $l$, we can assume 
that 
$$
m((l+4)g^*A|_H-K_G)
$$ 
is very ample, for some bounded natural number $m\ge 2$.

On the other hand, by assumption  
$$
C:=f^*A-(K_X+B)\sim_\R f^*(A-L)
$$ 
is nef, and we can write 
$$
m((l+2)f^*A-K_X)-G=2mf^*A+m(lf^*A-K_X)-G
$$
$$
\sim_\R (2m-2)f^*A+m(lf^*A-K_X)
$$
$$
\sim_\R K_X+B+C+(2m-3)f^*A+m(lf^*A-K_X).
$$
Thus by the Kawamata-Viehweg vanishing theorem,
$$
h^1(m((l+2)f^*A-K_X)-G)=0,
$$ 
hence the map 
$$
H^0(m((l+2)f^*A-K_X)) \to H^0(m((l+2)f^*A-K_X)|_G)
$$
is surjective. Recall that by the previous paragraph, 
$$
H^0(m((l+2)f^*A-K_X)|_G) = H^0(m((l+4)g^*A|_H-K_G)).
$$

Now since  
$$
m((l+4)g^*A|_H-K_G)
$$ 
is very ample, we can find a section in 
$$
H^0(m((l+4)g^*A|_H-K_G))
$$ 
vanishing at $x$ but not at $x'$ (and vice versa).
This in turn gives a section in 
$$
H^0(m((l+2)f^*A-K_X))
$$ 
vanishing at $x$ but not at $x'$ (and vice versa). 
Therefore, 
$$
|m((l+2)f^*A-K_X)|
$$ 
defines a birational map. Now replace $l$ with $l+2$.
\end{proof}

\subsection{Boundedness of volume}

The next statement and its proof are similar to \cite[Theorem 4.1]{Jiang} (also see \cite[4.1]{DiCerbo-Svaldi}).

\begin{prop}\label{l-bnd-cy-bndness-volume}
Let $d,r,l$ be natural numbers and $\varepsilon$ be a positive real number.
Then there exists  
a natural number $v$ depending only on $d,r,l,\varepsilon$ satisfying the following. 
If $(X,B)\to Z$ is a $(d,r,\varepsilon)$-Fano type fibration (as in \ref{d-FT-fib}), then 
$$
\vol(lf^*A-K_X)\le v.
$$ 
\end{prop}
\begin{proof}
We can assume that $lf^*A-K_X$ is big, otherwise 
$$
\vol(lf^*A-K_X)=0.
$$ 
Moreover, taking a $\Q$-factorialisation, we can assume $X$ is $\Q$-factorial. 
If $\dim Z=0$, then $X$ belongs to a bounded family  \cite[Corollary 1.4]{B-BAB} 
so the proposition follows. We then assume $\dim Z>0$.

Let $p$ be the largest integer such that $pf^*A-K_X$ is not big. 
Then $p<l$.  Let $H\in |A|$ be a general element and $G=f^*H$. Then 
$$
lf^*A-K_X=pf^*A-K_X+(l-p)f^*A\sim pf^*A-K_X+(l-p)G,
$$
so by \cite[Lemma 2.5]{Jiang}, 
$$
\vol(lf^*A-K_X)\le \vol(pf^*A-K_X)+d(l-p)\vol((lf^*A-K_X)|_G).
$$
Since $pf^*A-K_X$ is not big, 
$$
\vol(pf^*A-K_X)=0,
$$
 so it is then enough to bound 
$$
d(l-p)\vol(lf^*A-K_X)|_G)
$$
from above. 

Define 
$$
K_G+B_G=(K_X+B+G)|_G.
$$ 
Then $(G,B_G)$ is $\varepsilon$-lc,  $-K_G$ is big over $H$, 
$$
K_G+B_G\sim_\R g^*(L+A)|_H
$$ 
where $g$ denotes $G\to H$, and 
$$
2A|_H-(L+A)|_H
$$ 
is ample. 
  Thus $(G,B_G)\to H$ is a 
$(d-1, 2^{d-1}r, \varepsilon)$-Fano type fibration, hence  applying induction on dimension shows that 
$$
\vol((lf^*A-K_X)|_G)=\vol(((l+1)f^*A-(K_X+G))|_G)
$$
$$
=\vol((l+1)g^*A|_H-K_G)\le \vol((l+1)g^*2A|_H-K_G)
$$
is bounded from above. Therefore, it is enough to show that $l-p$ is bounded from above 
which is equivalent to showing that $p$ is bounded from below. 

By definition of $p$, 
$$
(p+1)f^*A-K_X
$$ 
is big. Thus we can find 
$$
0\le R\sim_\Q  (p+1)f^*A-K_X.
$$
Let $F$ be a general fibre of $f\colon X\to Z$. Then
$$
K_F+B_F:=(K_X+B)|_F\sim_\R 0
$$ 
and
$$
R_F:=R|_F\sim_\Q -K_X|_F=-K_F\sim_\R B_F.
$$ 

By \cite[Corollary 1.4]{B-BAB}, 
$F$ belongs to a bounded family of varieties as $(F,B_F)$ is $\varepsilon$-lc and $-K_F$ is big. 
We can then find a very ample divisor 
$J_F$ on $F$ with bounded degree $J_F^{\dim F}$ such that 
$$
J_F-R_F \sim_\R J_F-B_F \sim_\Q J_F+K_F  
$$ 
is ample. Therefore, by \cite[Theorem 1.8]{B-BAB} (=Theorem \ref{t-bnd-lct}), 
the pair $(F,B_F+tR_F)$ is klt for some real number $t\in(0,1)$ depending only 
on $d,J_F^{\dim F},\varepsilon$. In particular, letting 
$$
\Delta:=(1-t)B+tR ~~\mbox{ and}~~ \Delta_F:=\Delta|_F,
$$ 
we see that 
$$
(F,\Delta_F=(1-t)B_F+tR_F)
$$ 
is klt. Therefore, $(X,\Delta)$ is klt near the generic fibre of $f$.
By construction, 
$$
K_X+\Delta= K_X+(1-t)B+tR\sim_\R K_X+(1-t)B+t(p+1)f^*A-tK_X
$$
$$
= (1-t)(K_X+B)+t(p+1)f^*A \sim_\R f^*((1-t)L+t(p+1)A).
$$

Now by adjunction, we can write 
$$
K_X+\Delta\sim_\R f^*(K_Z+\Delta_Z+M_Z)
$$
where $\Delta_Z$ is the discriminant divisor and $M_Z$ is the moduli divisor (see \ref{fib-adj-setup} below for more 
details on adjunction). In particular,
$$
(1-t)L+t(p+1)A\sim_\R K_Z+\Delta_Z+M_Z.
$$
By \cite[Theorem 3.6]{B-compl}, $M_Z$ is pseudo-effective. 
On the other hand, if $\psi\colon Z'\to Z$ is a resolution, then $K_{Z'}+3d\psi^* A$ is big \cite[Lemma 2.46]{B-compl}, 
hence $K_Z+3dA$ is big. 
This implies that 
$$
K_Z+\Delta_Z+M_Z+3dA
$$
is big. But then  
$$
(1-t)L+t(p+1)A+3dA
$$
is big which in turn implies that 
$$
(1-t)A+t(p+1)A+3dA=(1-t)(A-L)+(1-t)L+t(p+1)A+3dA
$$
is big as $A-L$ is ample. Therefore, 
$$
(1-t)+t(p+1)+3d>0
$$ 
which implies that $p$ is bounded from below as required. 
\end{proof}

\subsection{Log birational boundedness}
We now turn to birational boundedness in the context of $(d,r,\varepsilon)$-Fano type fibrations.

\begin{prop}\label{l-bnd-cy-bir-bnd}
Let $d,r,l$ be natural numbers and $\varepsilon,\delta$ be positive real numbers.
Assume that Theorems \ref{t-log-bnd-cy-fib} and \ref{t-sing-FT-fib-totalspace} hold in 
dimension $d-1$. Then there exists a bounded set of couples $\mathcal{P}$  depending only 
on $d,r,l,\varepsilon,\delta$ satisfying the following. 
Assume that $(X,B)\to Z$ is a $(d,r,\varepsilon)$-Fano type fibration (as in \ref{d-FT-fib}) 
and that $\Lambda\ge 0$ is an $\R$-divisor such that 
\begin{itemize}
\item each non-zero coefficient of $\Lambda$ is $\ge \delta$, and 

\item $lf^*A-(K_X+\Lambda)$ is pseudo-effective.
\end{itemize}
Then there exist a couple
$(\overline{X}, \overline{\Sigma})$, a very ample divisor  $\overline{D}\ge 0$ on ${\overline{X}}$ and 
a birational map $\rho\colon \overline{X} \bir X /Z$ such that 
\begin{enumerate}
\item $(\overline{X}, \overline{\Sigma}+\overline{D})$ is log smooth and belongs to $\mathcal{P}$, 

\item $\overline{\Sigma}$ contains the exceptional divisors of $\rho$ union 
the birational transform of $\Supp \Lambda$, 

\item and if $N:=lf^*A-K_X$ is nef and $\overline{N}:=\rho^*N$ (defined as in \ref{ss-divisors}), 
then $\overline{D}-\overline{N}$ is ample. 
\end{enumerate}
\end{prop}
\begin{proof}
\emph{Step 1.}
\emph{In this step, we introduce some notation.}
Taking a $\Q$-factorialisation we can assume $X$ is $\Q$-factorial. 
Note that if the proposition holds for $d,r,l',\varepsilon,\delta$ where $l'\ge l$, then it also holds for $d,r,l,\varepsilon,\delta$.
Thus by Proposition \ref{l-bnd-cy-effective-bir}, 
perhaps after replacing $l$ with a bounded multiple,  we can assume that 
there exists a natural number 
$m$ depending only on $d,r,\varepsilon$ such that the linear system $|m(lf^*A-K_X)|$ defines a birational map. Pick 
$$
0\le M\sim m(lf^*A-K_X).
$$
Take a log resolution 
$\phi\colon W\to X$ of $(X,\Lambda+M)$ such that $|\phi^*M|$ decomposes as the 
sum of a free part $|F_W|$ plus fixed 
part $R_W$ (cf. \cite[Lemma 2.6]{B-compl}). 
Let $\Sigma_W$ be the sum of the reduced exceptional divisor of $\phi$ and the birational transform of $\Supp (\Lambda+M)$ 
plus a general element $G_W$ of $|\phi^*f^*A+F_W|$. 
Let $\Sigma,F,R,G$ be the pushdowns of $\Sigma_W,F_W,R_W, G_W$ to $X$.\\ 

\emph{Step 2.}
\emph{In this step, we show that 
$$
\vol(K_W+\Sigma_W+(4d+2)G_W)
$$ 
is bounded from above.}
From the definition of $\Sigma_W$ and the assumption that the non-zero coefficients of 
$\Lambda$ are $\ge \delta$, we can see that 
$$
\Sigma =\Supp (\Lambda+M)+G \le \frac{1}{\delta}\Lambda+M+G.
$$
Moreover, 
$$
G+R\sim f^*A+F+R\sim f^*A+M,
$$
 and by assumption  
$lf^*A-(K_X+\Lambda)$  is pseudo-effective. Taking all these into account and  
letting $p=4d+2$ and $q=\frac{1}{\delta}$ we then have 
$$
 \begin{array}{l l}
\vol(K_W+\Sigma_W+pG_W) & \le \vol(K_X+\Sigma+pG) \\
& \le \vol(K_X+q\Lambda+M+(p+1)G) \\
& \le \vol(K_X+q\Lambda+M+(p+1)f^*A +(p+1)M)\\
& = \vol(K_X+q\Lambda+(p+1)f^*A +(p+2)M)\\
& = \vol((1-q)K_X+q(K_X+\Lambda)+(p+1)f^*A+(p+2)M)\\
& \le \vol((1-q)K_X+qlf^*A+(p+1)f^*A+(p+2)M)\\
& = \vol((1-q)K_X+(ql+p+1)f^*A+(p+2)M)\\
& =\vol((ql+p+1)f^*A+(p+2)mlf^*A-((p+2)m+q-1)K_X).
\end{array}
$$
The latter volume is bounded from above by Proposition  \ref{l-bnd-cy-bndness-volume} 
as all the numbers $l,m,p,q$ are fixed.  Thus 
$\vol(K_W+\Sigma_W+pG_W)$ is bounded from above.\\ 

\emph{Step 3.}
\emph{In this step, we show that $(X,\Supp (\Lambda+M))$ is log birationally bounded.}
Since 
$$
G_W\sim \phi^*f^*A+ F_W,
$$ 
$|G_W|$ is base point free and it defines a birational contraction $W\to \overline{X}'$.
In particular,
$$
K_W+\Sigma_W+(p-1)G_W
$$ 
is big (cf. \cite[Lemma 2.46]{B-compl}), hence 
$$
\vol(G_W)\le \vol(K_W+\Sigma_W+pG_W)
$$ 
which implies that the left hand side is also bounded from above. 
 Moreover,  by \cite[Lemma 3.2]{HMX1}, 
$\Sigma_W\cdot G_W^{d-1}$ is bounded from above.
Therefore, if 
$ \overline{\Sigma}'$ is the pushdown of $\Sigma_W$, then $(\overline{X}', \overline{\Sigma}')$ 
is log bounded (this follows from \cite[Lemma 2.4.2(4)]{HMX1}). 
Also the induced map ${\overline{X}'} \bir Z$ is a morphism by the choice of $G_W$: indeed 
any curve contracted by $W\to \overline{X}'$ intersects $G_W$ trivially hence it intersects the pullback of 
$A$ trivially which means the curve is also contracted over $Z$.

Now we can take a log resolution $\overline{X}\to \overline{X}'$ of $(\overline{X}', \overline{\Sigma}')$ such that if 
$\overline{\Sigma}$ is the union of the exceptional divisors and the birational transform of $\overline{\Sigma}'$, 
then $(\overline{X}, \overline{\Sigma})$ is log smooth and log bounded. By construction, 
$\overline{\Sigma}$ contains the reduced exceptional divisor of the induce map $\rho\colon \overline{X}\bir X$ union the 
birational transform of $\Supp (\Lambda+M)$. This settles (1) and (2) of the proposition except that 
we need to add $\overline{D}$.\\

\emph{Step 4.}
\emph{In this step, we prove the existence of the very ample divisor  $\overline{D}$.}
Denote the induced map $\overline{X}\bir W$ by $\alpha$.
By construction, $G_{W}\sim 0/ \overline{X}'$, hence $\overline{G}:=\alpha^*G_W\sim 0/ \overline{X}'$ 
(pullback under $\alpha$ is defined as in \ref{ss-divisors}).
Moreover, $\overline{G}\le \overline{\Sigma}$. 
In particular, $(\overline{X}, \overline{G})$ is log bounded and $\overline{G}$ is big, 
hence we can find a bounded natural 
number $b$ and a very ample divisor $\overline{D}$
such that $b\overline{G}-\overline{D}$ is big. Then
$(\overline{X}, \overline{\Sigma}+\overline{D})$ is log bounded, hence it belongs to some 
fixed bounded set of couples $\mathcal{P}$.

From now on we assume that $N=lf^*A-K_X$ is nef. Then $M\sim mN$ is also nef.
We will show that we can choose $\overline{D}$ so that $\overline{D}-\overline{M}$ is ample where 
$\overline{M}=\rho^*M$. 
Let $\pi\colon V\to W$ and $\mu\colon V\to \overline{X}$ be a common resolution. 
Then 
$$
\begin{array}{l l}
\overline{D}^{d-1}\cdot \overline{M} = \mu^*\overline{D}^{d-1}\cdot \mu^* \overline{M}
& = \mu^*\overline{D}^{d-1}\cdot \pi^*\phi^*M\\
& \le \vol(\mu^*\overline{D}+\pi^*\phi^*M)\\
& \le  \vol(b\mu^*\overline{G}+\pi^*\phi^*M)\\
&=\vol(b\pi^* G_W+\pi^*\phi^*M)\\
& =\vol(bG_W+\phi^*M)\\
& \le \vol(bG+M)\\
& \le \vol(bf^*A+bM+M)\\
&=\vol(bf^*A+(b+1)mlf^*A-(b+1)mK_X)
\end{array}
$$
where to get the second equality we use the fact $\overline{M}=\mu_*(\pi^*\phi^*M)$ and to get the 
first inequality we have used the fact that $\mu^*\overline{D},\pi^*\phi^*M$ are both nef. 
Therefore, $\overline{D}^{d-1}\cdot \overline{M}$ is bounded from above as the latter volume is bounded 
from above by Proposition \ref{l-bnd-cy-bndness-volume}.
This implies that the coefficients of $\overline{M}$ are bounded from above. 

Now since $\Supp \overline{M}\le \overline{\Sigma}$, $(\overline{X}, \Supp \overline{M})$ 
 is log bounded. Thus replacing $\overline{D}$ we can assume that $\overline{D}-\overline{M}$ is ample.
Finally, note that  $\overline{N}\sim_\Q \frac{1}{m} \overline{M}$, hence  
$$
\overline{D}-\overline{N}\sim_\Q \overline{D}-\frac{1}{m} \overline{M}
=\frac{1}{m}(m\overline{D}- \overline{M})
$$ 
is ample.  
\end{proof}

\subsection{Lower bound on lc thresholds: special case}

We prove a special case of Theorem \ref{t-sing-FT-fib-totalspace} which is crucial for the 
rest of this section.

\begin{prop}\label{l-bnd-cy-bnd-lct-special}
Let $d,r,l$ be natural numbers and $\varepsilon$ be a positive real number.
Assume that Theorems \ref{t-log-bnd-cy-fib} and \ref{t-sing-FT-fib-totalspace}
hold in dimension $d-1$. Then there exists  
a positive real number $t$ depending only on $d,r,l,\varepsilon$ satisfying the following. 
Assume that $(X,B)\to Z$ is a $(d,r,\varepsilon)$-Fano type fibration (as in \ref{d-FT-fib}) and that 
$0\le P\sim_\R lf^*A$. Then $(X,B+tP)$ is klt. 
\end{prop}

We prove a lemma before proving the proposition.

\begin{lem}\label{l-bnd-cy-bnd-lct-special-I}
Let $d,r,l,n$ be natural numbers and $\varepsilon$ be a positive real number.
Assume that Theorems \ref{t-log-bnd-cy-fib} and \ref{t-sing-FT-fib-totalspace}
hold in dimension $d-1$. Then there exists  
a natural number $v$ depending only on $d,r,l,n,\varepsilon$ satisfying the following. 
Assume that $(X,B)\to Z$ is a $(d,r,\varepsilon)$-Fano type fibration (as in \ref{d-FT-fib}) and that 
\begin{itemize}
\item $0\le P\sim_\R lf^*A$,

\item $T$ is a prime divisor over $X$,

\item $(X,\Lambda)$ is lc over a neighbourhood of $z$ where $\Lambda\ge 0$ and 
$z$ is the generic point of the image of $T$ on $Z$,

\item $a(T,X,B)\le 1$ and $a(T,X,\Lambda)=0$, and

\item  $n(K_X+\Lambda)\sim (n+2)lf^*A$.
\end{itemize}
Then $\mu_TP\le v$.
\end{lem}
Note that $T$ may not be a divisor on $X$. Here $\mu_TP$ means the coefficient of $T$ in the pullback of $P$ to any resolution of $X$ on which $T$ is a divisor.
\begin{proof}
\emph{Step 1.}
\emph{In this step, we discuss log birational boundedness of $(X,\Supp \Lambda)$.} 
Taking a $\Q$-factorialisation, we can assume that $X$ is $\Q$-factorial. 
After running an MMP on $-K_X$ over $Z$, we can assume that $-K_X$ is nef over $Z$. 
By Lemma \ref{l-bnd-cy-numerical-bndness}, $qf^*A-K_X$ is globally nef for some bounded natural number $q$. 
By Lemma \ref{l-bnd-cy-bir-bnd}, $(X,\Supp \Lambda)$ is log birationally bounded,
that is, there exist a couple $(\overline{X}, \overline{\Sigma})$, a very ample divisor $\overline{D}$, and  
a birational map $\rho\colon \overline{X}\bir X/Z$ such that 
\begin{itemize}
\item $(\overline{X}, \overline{\Sigma}+\overline{D})$ is log smooth and belongs to a bounded set of couples $\mathcal{P}$, 

\item $\overline{\Sigma}$ contains the exceptional divisors of $\rho$ union 
the birational transform of $\Supp \Lambda$, 

\item and if  $N:=qf^*A-K_X$ and $\overline{N}=\rho^*N$, then $\overline{D}-\overline{N}$ is ample. 
\end{itemize}
Let $K_{\overline{X}}+\overline{B}$ and $K_{\overline{X}}+\overline{\Lambda}$ be the 
pullbacks of $K_X+B$ and $K_X+\Lambda$, respectively. Since $(X,\Lambda)$ is lc over $z$ 
and $K_X+\Lambda\sim_\Q 0/Z$, 
$(\overline{X}, \overline{\Lambda})$ is  sub-lc over $z$. Moreover, from 
$
a(T,{X}, \Lambda)=0
$ 
we get 
$
a(T,\overline{X}, \overline{\Lambda})=0.
$ 
 But then since 
$\Supp  \overline{\Lambda} \subseteq \overline{\Sigma}$, we have $\overline{\Lambda} \le \overline{\Sigma}$ 
over $z$, hence $T$ is also an lc place of $(\overline{X}, \overline{\Sigma})$, 
that is, 
$$
a(T,\overline{X}, \overline{\Sigma})=0.
$$\ 

\emph{Step 2.}
\emph{In this step, we study numerical properties of $\overline{D}$.}
Since $(\overline{X}, \overline{\Sigma})$ is log bounded, we can assume that 
$\overline{D}-\overline{\Sigma}$
is ample. Moreover, by adding a general element of $|(n+2)f^*A|$ to $\Lambda$ we can assume that some element of 
$|(n+2)\overline{f}^*A|$ is a component of $\Sigma_{\overline{X}}$ where $\overline{f}$ denotes $\overline{X}\to Z$; 
this requires replacing $l$ with 
$l+n$ to preserve the condition 
$$
n(K_X+\Lambda)\sim (n+2)lf^*A,
$$ 
and replacing $P$ accordingly. 
Thus we can also assume that 
$$
\overline{D}-(n+2)\overline{f}^*A
$$ 
is ample and that $\overline{D}-\overline{P}$ 
is ample where $\overline{P}=\rho^*P$.

Now
$$
\overline{D}-(K_{\overline{X}}+\overline{B})
\sim_\R \overline{D}-\overline{f}^*A+\overline{f}^*A-(K_{\overline{X}}+\overline{B})
$$
$$
\sim_\R \overline{D}-\overline{f}^*A+ \overline{f}^*(A-L)
$$ 
is ample. In addition we can assume 
 $\overline{D}+K_{\overline{X}} $ is ample as well, hence replacing $\overline{D}$ with  $2\overline{D}$
we can assume that  $\overline{D}-\overline{B}$ is ample.\\ 

\emph{Step 3.}
\emph{In this step, we prove the lemma assuming that the coefficients of $\overline{B}$ are bounded from below.} 
That is, assume that the coefficients of $\overline{B}$ are $\ge p$ for some fixed integer $p$.
Under this assumption, there is $c\in (0,1)$ depending only on $p$ such that 
$$
\overline{\Delta}:=c\overline{B}+(1-c)\overline{\Sigma}\ge 0
$$
because the components of $\overline{B}$ with negative coefficients are exceptional over $X$, hence 
are components of $\overline{\Sigma}$.
In particular, since $(\overline{X},\overline{B})$ is sub-$\varepsilon$-lc, 
$(\overline{X},\overline{\Delta})$ is a $c\varepsilon$-lc pair. Moreover, 
$$
a(T,\overline{X},\overline{\Delta})=ca(T,\overline{X},\overline{B})+(1-c)a(T,\overline{X},\overline{\Sigma})
=ca(T,\overline{X},\overline{B})<1.
$$
In addition,
$$
\overline{D}-\overline{\Delta}=\overline{D}-c\overline{B}-(1-c)\overline{\Sigma}
=c(\overline{D}-\overline{B})+(1-c)(\overline{D}-\overline{\Sigma})
$$
is ample. 

Now applying \cite[Theorem 1.8]{B-BAB} (=Theorem \ref{t-bnd-lct}), we deduce that 
$(\overline{X},\overline{\Delta}+t\overline{P})$ is klt for some real number $t>0$ bounded away from zero. 
Therefore, $\mu_T\overline{P}$ is bounded from above because 
$$
0\le a(T,\overline{X},\overline{\Delta}+t\overline{P})
=a(T,\overline{X},\overline{\Delta})-t\mu_T\overline{P}<1-t\mu_T\overline{P}.
$$
But $\mu_TP=\mu_T\overline{P}$ because the pullback of $P$ and $\overline{P}$ to any common resolution of $X,\overline{X}$ coincide in view of $P\sim_\Q 0/Z$.\\ 

\emph{Step 4.}
\emph{Finally, it is enough to show that the coefficients of $\overline{B}$ are bounded from below.} 
Define $K_{\overline{X}}+\overline{E}=\rho^*K_X$. It is enough to show that 
the coefficients of $\overline{E}$ are bounded from below because $\overline{E}\le \overline{B}$. Write 
$\overline{E}$ as the difference $\overline{E}^+ - \overline{E}^-$ where $\overline{E}^+, \overline{E}^-\ge 0$ 
have no common components. 
Observe that 
$$
\overline{N}=\rho^*(qf^*A-K_X)= q\overline{f}^*A-(K_{\overline{X}}+\overline{E})
=q\overline{f}^*A-(K_{\overline{X}}+\overline{E}^+)+\overline{E}^-,
$$
hence 
$$
2\overline{D}-\overline{E}^-\sim_\R \overline{D}-\overline{N}+\overline{D}-(K_{\overline{X}}+\overline{E}^+)+q\overline{f}^*A.
$$
By Step 1, $\overline{D}-\overline{N}$ is ample and $\overline{E}^+\le \overline{\Sigma}$.
Replacing $\overline{D}$ with a multiple we can assume that $\overline{D}-(K_{\overline{X}}+\overline{E}^+)$ is ample.
Thus $2\overline{D}-\overline{E}^-$ is ample which implies that $\overline{D}^{d-1}\cdot\overline{E}^-$ is bounded 
from above, hence the coefficients of 
$\overline{E}^-$ are bounded from above which in turn implies that the coefficients of $\overline{E}$ 
are bounded from below as required.
\end{proof}

\begin{proof}[Proof of Proposition \ref{l-bnd-cy-bnd-lct-special}]
\emph{Step 1.}
\emph{In this step, we will translate the problem into showing that the multiplicity of $P$ along certain 
divisors is bounded from above.} First we can assume 
$P\neq 0$ otherwise the statement is trivial. In particular,  $\dim Z>0$.
Taking a $\Q$-factorialisation we can assume $X$ is $\Q$-factorial.
Pick a small $\varepsilon'\in (0,\varepsilon)$. Let $s$ be the $\varepsilon'$-lc threshold of $P$ with respect to 
$(X,B)$, that is, $s$ is the largest number such that $(X,B+sP)$ is $\varepsilon'$-lc. It is enough to show that 
$s$ is bounded from below away from zero. In particular, we can assume $s<1$.

There is a prime divisor $T$ over $X$ with log discrepancy
$$
a(T,X,B+sP)=\varepsilon'<1.
$$ 
Since $P$ is vertical over $Z$, $T$ is vertical over $Z$.
It is enough to show that $\mu_TP$, the coefficient of $T$ in the pullback of $P$ on any resolution, 
is bounded from above because 
$$
s\mu_TP=a(T,X,B)-a(T,X,B+sP)\ge \varepsilon-\varepsilon'.
$$
We devote the rest of the proof to showing that $\mu_TP$ is bounded from above.\\ 

\emph{Step 2.}
\emph{In this step, we apply induction and reduce to the case when $T$ maps to a 
closed point on $Z$.} By the choice of $P$, 
$$
K_X+B+sP\sim_\R f^*(L+slA),
$$
and since $s<1$, 
$$
(l+1)A-(L+slA)
$$ 
is ample. Thus 
replacing $B$ with $B+sP$, replacing $A$ with $(l+1)A$ (and replacing $r$ accordingly), 
and replacing $\varepsilon$ with $\varepsilon'$, we can assume 
that $\varepsilon$ is sufficiently small and that  $a(T,X,B)=\varepsilon$ (we will not use $s$ any more). 
Extracting $T$ we can also assume $T$ is a divisor on $X$. 
Our goal still is to show that $\mu_TP$ is bounded from above.

Take a hyperplane section $H\sim A$ of $Z$ and let $G=f^*H$. Consider
$$
K_G+B_G:=(K_X+B+G)|_G
$$ 
and $P_G:=P|_G$. Then $(G,B_G)$ is $\varepsilon$-lc, $-K_G$ is big over $H$, 
$$
K_G+B_G\sim_\R g^*(L+A)|_H
$$ 
where $g$ denotes $G\to H$, and $2A|_H-(L+A)|_H$ is ample. 
Thus $(G,B_G)\to H$ is a $(d-1, 2^{d-1}r, \varepsilon)$-Fano type fibration.
Moreover, $2P_G\sim_\R lg^*2A|_H$.
Applying induction on dimension we find a real number $u>0$ depending only on $d,r,l,\varepsilon$ 
such that $(G,B_G+uP_G)$ is klt. Then by inversion of adjuction  \cite[Theorem 5.50]{kollar-mori} (which is 
stated for $\Q$-divisors but also holds for $\R$-divisors),
the pair $(X,B+G+uP)$ is plt near $G$. In particular, if the image of $T$ on $Z$ is positive-dimensional, 
then $T$ intersects $G$, then $\mu_TP$ is bounded from above. 
Therefore, we can assume that the image of $T$ on $Z$ is a closed point $z$.\\ 

\emph{Step 3.}
\emph{In this step, we finish the proof by applying Lemma \ref{l-bnd-cy-bnd-lct-special-I}.}
Let $\Delta=(1-\varepsilon)T$ and let $\Theta=T$ (we introduce this notation because we will apply the  arguments of this step later in the proof of \ref{l-bnd-cy-bnd-klt-compl-1}). 
Since $T$ is vertical over $Z$, $-(K_X+\Theta)$ is big over $Z$.
Run an MMP on $-(K_X+\Theta)$ over $Z$ 
and let $X'$ be the resulting model. We denote the pushdown of each divisor $D$ to $X'$ 
by $D'$. Then $-(K_{X'}+\Theta')$ is nef and big over $Z$.
By construction, $-\varepsilon T'\le B'-\Theta'$, hence since $T$ is mapped to a closed point on $Z$, 
$(f^*A)'+B'-\Theta'$ is pseudo-effective. Thus   
by Proposition \ref{l-bnd-cy-numerical-bndness}, we can assume that $(lf^*A)'-(K_{X'}+\Theta')$ is nef  after replacing $l$. Since $(lf^*A)'-(K_{X'}+\Theta')$ is nef and relatively big over $Z$ and since $A$ is ample, 
$$
((l+1)f^*A)'-(K_{X'}+\Theta')
$$ 
is nef and big by Lemma \ref{l-div-pef-on-FT}. So replacing $l$ with $l+1$ we can assume that $(lf^*A)'-(K_{X'}+\Theta')$ is nef and big.
 
 On the other hand, $(X',\Theta'-\varepsilon T')$ is klt as $\Theta'-\varepsilon T'\le B'$. 
Since $\varepsilon$ is assumed to be sufficiently small, by the ACC for 
lc thresholds \cite[Theorem 1.1]{HMX2}, $(X',\Theta')$ is lc.
Note that $T$ is not contracted over $X'$: assume not; then since $X\bir X'$ is an MMP on $-(K_{X}+\Theta)$ we get  
$$
0=a(T,X,\Theta)\ge a(T,X',\Theta')=a(T,X',\Delta')
$$
which implies that $(X',\Delta')$ is not klt contradicting the facts that $(X',B')$ is klt and $\Delta'\le B'$. 
Now $T'$ is clearly a non-klt centre of $(X',T')$. 

Then applying Theorem \ref{t-bnd-comp-lc-global} (by taking $B=\Theta'$, $M=(lf^*A)'$ and 
$S=T'$), there exist a 
bounded natural number $n$ and $\Lambda'\ge \Theta'$ such that $(X',\Lambda')$ is lc over $z$ and 
$$
n(K_{X'}+\Lambda')\sim (n+2)l(f^*A)'.
$$ 
Since $X\bir X'$ is an MMP on $-(K_X+\Theta)$ over $Z$, 
taking $K_X+\Lambda$ to be the crepant pullback of 
$K_{X'}+\Lambda'$ to $X$ we get $\Lambda\ge \Theta\ge \Delta$ such that $(X,\Lambda)$ is lc over $z$ and 
$$
n(K_{X}+\Lambda)\sim (n+2)lf^*A.
$$ 
Finally, apply Lemma \ref{l-bnd-cy-bnd-lct-special-I} 
to deduce that $\mu_TP$ is bounded.
\end{proof}

\subsection{Bounded klt complements}
In this subsection, we treat Theorem \ref{t-bnd-comp-lc-global-cy-fib} inductively. 
We first consider a weak version.

\begin{prop}\label{l-bnd-cy-bnd-klt-compl-1}
Let $d,r$ be natural numbers, $\varepsilon$ be a positive real number, and $\mathfrak{R}\subset [0,1]$ 
be a finite set of rational numbers.
Assume that Theorems \ref{t-log-bnd-cy-fib} and \ref{t-sing-FT-fib-totalspace}
hold in dimension $d-1$. Then there exist  
natural numbers $n,m$ depending only on $d,r,\varepsilon,\mathfrak{R}$ satisfying the following. 
Assume that $(X,B)\to Z$ is a $(d,r,\varepsilon)$-Fano type fibration (as in \ref{d-FT-fib}) and that 
\begin{itemize}
\item we have $0\le \Delta\le B$ with coefficients in $\mathfrak{R}$, and 

\item $-(K_X+\Delta)$ is big over $Z$. 
\end{itemize}
Then for each point $z\in Z$ there is a $\Q$-divisor $\Lambda\ge \Delta$ such that 
\begin{itemize}
\item $(X,\Lambda)$ is lc over $z$, and

\item $n(K_X+\Lambda)\sim mf^*A$. 
\end{itemize} 
\end{prop}

\begin{proof}
\emph{Step 1.}
\emph{In this step, we create singularities over $z$.}
We can assume $\varepsilon<1$.
It is enough to prove the proposition with $z$ replaced by any closed point $z'$ in the closure $\bar{z}$ because 
any open neighbourhood of $z'$ contains $z$. Thus from now on we assume that $z$ is a closed point.  
Taking a $\Q$-factorialisation, we can assume $X$ is $\Q$-factorial. 
We can assume that $\dim Z>0$ otherwise we apply \cite[Theorem 1.7]{B-compl} (when $\dim Z=0$ we first run an MMP on $-(K_X+\Delta)$ to make $-(K_X+\Delta)$ nef. ).

Consider the sub-linear system $V_z$ of $|f^*A|$ consisting of elements 
containing the fibre $f^{-1}\{z\}$, and pick $P$ in $V_z$.
Since $A$ is very ample, $V_z$ is base point free outside $f^{-1}\{z\}$. 
Replacing $A$ with $2A$ we can assume that $\dim V_z>0$.

Let $p$ be a natural number such that $\frac{1}{p}<1-\varepsilon$. 
Pick distinct general elements $M_1,\dots,M_{p(d+1)}$ in $V_z$ and let 
$$
M=\frac{1}{p}(M_1+\dots+M_{p(d+1)}).
$$ 
Then $(X,B+M)$ is $\varepsilon$-lc  outside $f^{-1}\{z\}$ by generality of the $M_i$ and the assumption 
$\frac{1}{p}<1-\varepsilon$.  On the other hand, $(X,B+M)$ is not lc at any point of 
 $f^{-1}\{z\}$ by \cite[Theorem 18.22]{Kollar-flip-abundance}.\\
 
\emph{Step 2.}
\emph{In this step, we reduce the problem to the situation when there is a prime divisor 
$T$ on $X$ mapping to $z$ with $a(T,X,B)=\varepsilon$ sufficiently small.} 
Now pick a sufficiently small rational number $\varepsilon'\in (0,\varepsilon)$ and 
let $u$ be the largest number such that $(X,B+uM)$ is $\varepsilon'$-lc. There is a prime divisor $T$ over $X$ 
such that 
$$
a(T,X,B+uM)=\varepsilon'.
$$
As $(X,B+M)$ is not lc near $f^{-1}\{z\}$, $u<1$. 
Since $(X,B+uM)$ is $\varepsilon$-lc outside $f^{-1}\{z\}$ and since $\varepsilon'<\varepsilon$, 
the centre of $T$ on $X$ is contained in $f^{-1}\{z\}$. On the other hand, it is clear that 
$$
K_X+B+uM\sim_\R f^*(L+u(d+1)A).
$$ 
  
Replacing $\varepsilon$ with $\varepsilon'$ and replacing $B$ with $B+uM$ (and replacing $A,r$ accordingly) 
we can assume that 
$\varepsilon$ is sufficiently small and that there is a prime divisor $T$ over $X$ mapping to $z$ with $a(T,X,B)=\varepsilon$. 
Extracting $T$ we can assume it is a divisor on $X$; if $T$ is not exceptional over the 
original $X$, we increase the coefficient of $T$ in $\Delta$ to $1-\varepsilon$; but if $T$ is exceptional over the original $X$, 
then we let $\Delta$ be the birational transform of the original $\Delta$ plus $(1-\varepsilon)T$ (we then add $1-\varepsilon$ to $\mathfrak{R}$ so that the coefficients of $\Delta$ are still in $\mathfrak{R}$). 
The bigness of $-(K_X+\Delta)$ over $Z$ is preserved as $T$ is vertical over $Z$, and the condition $\Delta\le B$ is also preserved.\\ 

\emph{Step 3.}
\emph{In this step, we find a bounded complement of $K_X+\Delta$ using Theorem \ref{t-bnd-comp-lc-global}.}
Let $\Theta$ be the same as $\Delta$ except that we increase the coefficient of $T$ to $1$.
Adding $1$ to $\mathfrak{R}$ we can assume that the coefficients of $\Theta$ are in $\mathfrak{R}$. 
Then following the same arguments as in Step 3 of the proof of \ref{l-bnd-cy-bnd-lct-special} using the same notation shows that there exist bounded natural numbers $n,m$ and $\Lambda\ge T$ such that $(X,\Lambda)$ is lc over $z$ and 
$n(K_{X}+\Lambda)\sim mf^*A$.
\end{proof}

Now we strengthen the previous statement by replacing lc over $z$ with klt over $z$. 

\begin{prop}\label{l-bnd-cy-bnd-klt-compl-2}
Let $d,r$ be natural numbers, $\varepsilon$ be a positive real number, and $\mathfrak{R}\subset [0,1]$ 
be a finite set of rational numbers.
Assume that Theorems \ref{t-log-bnd-cy-fib}, \ref{t-sing-FT-fib-totalspace}, and 
\ref{t-bnd-comp-lc-global-cy-fib} hold in dimension $d-1$. Then there exist  
natural numbers $n,m$ depending only on $d,r,\varepsilon,\mathfrak{R}$ satisfying the following. 
Assume that $(X,B)\to Z$ is a $(d,r,\varepsilon)$-Fano type fibration (as in \ref{d-FT-fib}) and that 
\begin{itemize}
\item we have $0\le \Delta\le B$ with coefficients in $\mathfrak{R}$, and 

\item $-(K_X+\Delta)$ is big over $Z$. 
\end{itemize}
Then for each point $z\in Z$ there is a $\Q$-divisor $\Lambda\ge \Delta$ such that 
\begin{itemize}
\item $(X,\Lambda)$ is klt over $z$, and

\item $n(K_X+\Lambda)\sim mf^*A$. 
\end{itemize} 
\end{prop}
\begin{proof}
\emph{Step 1.}
\emph{In this step, we modify $B$ and introduce a divisor $\tilde\Delta$.}
We can assume that $\dim Z>0$ otherwise we apply \cite[Corollary 1.2]{B-compl} which shows that 
$(X,\Delta)$ is log bounded. Moreover, it is enough to prove the proposition by replacing $z$ 
with any closed point $z'$ in $\bar{z}$ because 
$(X,\Lambda)$ being klt over $z'$ implies that it is klt over $z$. Thus from now on we assume that $z$ is a closed point. 
Then as $A$ is very ample we can find $P\in |f^*A|$ containing $f^{-1}\{z\}$.

 By Proposition \ref{l-bnd-cy-bnd-lct-special}, there is a rational number 
$t>0$ depending only on $d,r,\varepsilon$ such that  $(X,B+2tP)$ is lc. Since 
$$
B+tP=\frac{1}{2}B+\frac{1}{2}(B+2tP),
$$
the pair $(X,B+tP)$ is $\frac{\varepsilon}{2}$-lc. Moreover, 
$$
K_X+B+tP\sim_\R f^*(L+tA).
$$
Thus replacing $B$ with $B+tP$ 
 and replacing $\varepsilon$ with $\frac{\varepsilon}{2}$, 
we can assume that 
$$
B\ge \tilde{\Delta}:=\Delta +tP
$$ 
for some fixed rational number $t\in (0,1)$ (here we can replace 
$A$ with $2A$  to ensure that $f^*A-(K_X+B)$ is still nef, and then replace $r$ accordingly). 
Since $P$ is integral and $t$ is fixed, the coefficients of $\tilde{\Delta}$ belong to 
a fixed finite set, so expanding $\mathfrak{R}$ we can assume they belong to $\mathfrak{R}$.\\

\emph{Step 2.}
\emph{In this step, we reduce the proposition to existence of a special lc complement.}
Assume that there exist bounded natural numbers $n,m$ and a $\Q$-divisor $\Lambda\ge \tilde{\Delta}$ 
such that 
\begin{enumerate}
\item $(X,\Lambda)$ is lc over $z$, 

\item the non-klt locus of $(X,\Lambda)$ is mapped to a finite set of closed points on $Z$, and 

\item that $n(K_X+\Lambda)\sim mf^*A$. 
\end{enumerate}

Assume that $Q\in |f^*A|$ is general and let 
$$
\Lambda':=\Lambda-tP+tQ.
$$ 
By (2), any non-klt centre of $(X,\Lambda)$ intersecting $f^{-1}\{z\}$ is actually contained in 
$f^{-1}\{z\}$. Thus since $P$ contains $f^{-1}\{z\}$, $(X,\Lambda')$ is klt over $z$. Moreover, 
$\Lambda'\ge \Delta$, and perhaps after replacing $n,m$ with a bounded multiple 
we have   
$$
n(K_X+\Lambda')=n(K_X+\Lambda-tP+tQ)\sim n(K_X+\Lambda)\sim m f^*A.
$$
Therefore, it is enough to find $n,m,\Lambda$ as in (1)-(3). 
At this point we replace $\Delta$ with $\tilde{\Delta}$. The bigness of $-(K_X+\Delta)$ over $Z$ 
is preserved as $P$ is vertical.\\ 

\emph{Step 3.}
\emph{In this step, we find a bounded lc complement of $K_X+\Delta$ and study it.}
After taking a $\Q$-factoriallisation of $X$ and 
running an MMP on $-(K_X+\Delta)$ over $Z$ we can assume that $-(K_X+\Delta)$ is nef over $Z$. 
Applying Proposition \ref{l-bnd-cy-numerical-bndness}, there is a bounded natural number $l$ such that 
$lf^*A-(K_X+\Delta)$ is nef globally. Replacing $l$ with $l+1$ we can assume 
$lf^*A-(K_X+\Delta)$ is nef and big.

By Proposition \ref{l-bnd-cy-bnd-klt-compl-1}, there exist bounded natural numbers $n,m$ and a 
$\Q$-divisor $\Lambda\ge \Delta$ such that $(X,\Lambda)$ is lc over $z$ and 
$$
n(K_{X}+\Lambda)\sim mf^*A.
$$
 Then 
$$
n(\Lambda-\Delta)=n(K_X+\Lambda)-n(K_X+\Delta)\sim mf^*A-n(K_X+\Delta),
$$
hence
 $$
 n(\Lambda-\Delta)\in |mf^*A-n(K_X+\Delta)|.
 $$
Multiplying $n,m$ by a bounded number we can assume that $n\Delta$ is integral.

Now by adding a general member of $|2lf^*A|$ to $\Lambda$ and replacing $m$ with $m+2nl$ 
to preserve $n(K_{X}+\Lambda)\sim mf^*A$, we can assume that $m-1\ge l(n+1)$, hence  
$$
(m-1)f^*A-(n+1)(K_X+\Delta)
$$
is nef and big.\\ 

\emph{Step 4.}
\emph{In this step, we consider the restriction of $|mf^*A-n(K_X+\Delta)|$ to a general 
member of $|f^*A|$.} Let $H$ be a general member of $|A|$ and let $G=f^*H$. Then 
$$
mf^*A-n(K_X+\Delta)-G\sim K_X+\Delta+(m-1)f^*A-(n+1)(K_X+\Delta).
$$
 Thus 
$$
H^1(mf^*A-n(K_X+\Delta)-G)=0
$$ 
by the Kawamata-Viehweg vanishing theorem, hence the restriction map 
$$
H^0(mf^*A-n(K_X+\Delta)) \to H^0((mf^*A-n(K_X+\Delta))|_G) 
$$
is surjective. Note that for any Weil divisor $D$ on $X$, we have 
$$
\mathcal{O}_X(D)\otimes\mathcal{O}_G\simeq \mathcal{O}_G(D|_G)
$$
by the choice of $G$ (see \cite[2.41]{B-compl}). This is used to get the above surjectivity.\\ 

\emph{Step 5.}
\emph{In this step, we consider complements on $G$.}
Define 
$$
K_G+B_G=(K_X+B+G)|_G
$$ 
and
$$
K_G+\Delta_G=(K_X+\Delta+G)|_G.
$$ 
Then as we have seen several times in this section, $(G,B_G)\to H$ is a $(d-1,r',\varepsilon)$-Fano type 
fibration for some fixed $r'$. Moreover, $\Delta_G\le B_G$, the coefficients of $\Delta_G$ are in $\mathfrak{R}$, 
and $-(K_G+\Delta_G)$ is big over $H$. 

Since we are assuming Theorem \ref{t-bnd-comp-lc-global-cy-fib} in dimension $d-1$, 
there exist bounded natural numbers $p,q$ and there is a $\Q$-divisor $\Lambda_{G}'\ge \Delta_G$ such that 
$(G,\Lambda_G')$ is klt and 
$$
p(K_G+\Lambda_G')\sim qg^*A|_H
$$
 where $g$ denotes the morphism $G\to H$. 
Replacing both $n$ and $p$ with $np$  and then replacing $m$ and $q$ with $mp$ and $nq$, respectively, 
we can assume that $n=p$. Next if $q<m+n$, then we increase $q$ to $m+n$ by adding $\frac{1}{n}D_G$ 
to $\Lambda_G'$ where $D_G$ is a general element of $(m+n-q)g^*A|_H$. 
If $q\ge m+n$, we similarly increase $m$ to $q-n$ by modifying $\Lambda$ 
so that we can again assume $q=m+n$. Thus we now have 
$$
n(K_G+\Lambda_G')\sim (m+n)g^*A|_H.
$$
Note that in the process, the  inequality $m-1\ge l(n+1)$
of Step 3 is preserved, so the surjectivity of Step 4 still holds.\\

\emph{Step 6.}
\emph{In this step, we finish the proof.}
By construction,
$$
nR_G:=n(\Lambda_G'-\Delta_G)\in |(m+n)g^*A|_H-n(K_G+\Delta_G)|,
$$
and 
$$
(G,\Lambda_G'=\Delta_G+R_G)
$$ 
is klt. Thus if we replace $nR_G$ with any general element of
 $$
 |(m+n)g^*A|_H-n(K_G+\Delta_G)|,
 $$ 
then the pair $(G,\Delta_G+R_G)$ is still klt. On the other hand,
$$  
\begin{array}{l l}
 (m+n)g^*A|_H-n(K_G+\Delta_G) &= (m+n)g^*A|_H-n(K_X+\Delta+G)|_G\\
 & =((m+n)f^*A-nG-n(K_X+\Delta))|_G\\
 & \sim (mf^*A-n(K_X+\Delta))|_G.
\end{array}
$$
  Thus, by the surjectivity in Step 4, a general element 
$$
nR\in |mf^*A-n(K_X+\Delta)|
$$
restricts to a general element 
 $$
nR_G\in |(m+n)g^*A|_H-n(K_G+\Delta_G)|.
 $$ 
 
Now in view of 
$$
K_G+\Delta_G+R_G=(K_X+\Delta+R+G)|_G
$$ 
and inversion of adjunction \cite[Theorem 5.50]{kollar-mori},
the pair $(X,\Delta+R+G)$ is plt near $G$, hence $(X,\Delta+R)$ is klt near $G$. 
Therefore, replacing $\Lambda$ with $\Delta+R$ we can assume that $(X,\Lambda)$ is klt near 
$G$. In other words, the non-klt locus of 
$(X,\Lambda)$ is mapped to a finite set of closed points of $Z$.
Note that $(X,\Lambda)$ is still lc over $z$ because 
 $$
 n(\Lambda-\Delta)\in |mf^*A-n(K_X+\Delta)|
 $$ 
and because 
$$
nR\in |mf^*A-n(K_X+\Delta)|
$$ 
is a general element. Thus we have satisfied the conditions (1)-(3) of Step 2.
\end{proof}

\begin{lem}\label{l-bnd-cy-bnd-klt-compl-induction}
Assume that Theorems \ref{t-log-bnd-cy-fib}, \ref{t-sing-FT-fib-totalspace}, and \ref{t-bnd-comp-lc-global-cy-fib} 
hold in dimension $d-1$. Then Theorem \ref{t-bnd-comp-lc-global-cy-fib} holds in dimension $d$.
\end{lem}
\begin{proof}
 By Proposition \ref{l-bnd-cy-bnd-klt-compl-2}, there exist 
natural numbers $n,m$ depending only on $d,r,\varepsilon, \mathfrak{R}$ such that for each point $z\in Z$ 
 there is a $\Q$-divisor $\Gamma\ge \Delta$ such that 
\begin{itemize}
\item $(X,\Gamma)$ is klt over some neighbourhood $U_z$ of $z$, and

\item $n(K_X+\Gamma)\sim mf^*A$. 
\end{itemize} 
We can find finitely many closed points $z_1,\dots,z_p$ in $Z$ such that the corresponding open sets $U_{z_i}$ cover 
$Z$. For each $z_i$ let $\Gamma_i$ be the corresponding boundary as above.

From 
$$
n(\Gamma_i-\Delta)= n(K_X+\Gamma_i)-n(K_X+\Delta)\sim mf^*A -n(K_X+\Delta)
$$ 
we get
$$
n(\Gamma_i-\Delta) \in |mf^*A-n(K_X+\Delta)|.
$$
Therefore, if $nR$ is a general member of $|mf^*A-n(K_X+\Delta)|$ and if we let $\Lambda:=\Delta+R$, 
then 
\begin{itemize}
\item $(X,\Lambda)$ is klt over $U_{z_i}$,  and 

\item $n(K_X+\Lambda)\sim mf^*A$. 
\end{itemize} 
Finally, since we have only finitely many open sets $U_{z_i}$ involved, $(X,\Lambda)$ is klt everywhere. 
\end{proof}

\subsection{A special case of boundedness of Fano type fibrations}

We treat a special case of Theorem \ref{t-log-bnd-cy-fib} inductively.

\begin{lem}\label{l-bnd-cy-fib-ample-case}
Let $d,r$ be natural numbers, $\varepsilon$ be a positive real number, and $\mathfrak{R}\subset [0,1]$ be a 
finite set of rational numbers. Assume that Theorems \ref{t-log-bnd-cy-fib}, \ref{t-sing-FT-fib-totalspace}, 
and \ref{t-bnd-comp-lc-global-cy-fib} hold in dimension $d-1$. Consider the set of all
$(d,r,\varepsilon)$-Fano type fibrations $(X,B)\to Z$ (as in \ref{d-FT-fib}) and $\R$-divisors $0\le \Delta\le B$ such that  
\begin{itemize}
\item the coefficients of $\Delta$ are in $\mathfrak{R}$, and

\item $-(K_X+\Delta)$ is ample over $Z$. 
\end{itemize} 
Then the set of such $(X,\Delta)$ is log bounded.
\end{lem}

\begin{proof}
By Lemma \ref{l-bnd-cy-bnd-klt-compl-induction}, our assumptions imply Theorem \ref{t-bnd-comp-lc-global-cy-fib} 
in dimension $d$, hence there exist 
natural numbers $n,m$ depending only on $d,r,\varepsilon,\mathfrak{R}$ and a $\Q$-divisor $\Lambda\ge \Delta$ such that 
\begin{itemize}
\item $(X,\Lambda)$ is klt, and  

\item $n(K_X+\Lambda)\sim mf^*A$. 
\end{itemize} 
We have 
$$
n(\Lambda-\Delta) \in |mf^*A-n(K_X+\Delta)|.
$$
Increasing $m$ (by adding to $\Lambda$ appropriately) and applying Proposition \ref{l-bnd-cy-numerical-bndness}, 
we can assume that $mf^*A-n(K_X+\Delta)$ is nef and that $l:=\frac{m}{n}$ is a natural number. 
Since $-(K_X+\Delta)$ is ample over $Z$, replacing $m$ with $2m$ (which then replaces 
$l$ with $2l$), we can assume that $mf^*A-n(K_X+\Delta)$ is ample.
In particular,  $\Lambda-\Delta$ is ample. 

Since $(X,\Lambda)$ is klt and $n(K_X+\Lambda)$ is Cartier, $(X,\Lambda)$ is  $\frac{1}{n}$-lc. 
Pick a small $t>0$ such that 
$$
(X,\Theta:=\Lambda+t(\Lambda-\Delta))
$$
is $\frac{1}{2n}$-lc. Here $t$ depends on $(X,\Lambda)$. 
Then
$$
K_X+\Theta=K_X+\Lambda+t(\Lambda-\Delta)\sim_\Q l f^*A+t(\Lambda-\Delta)
$$
is ample. In addition, since $\Supp(\Lambda-\Delta)\subseteq \Lambda$ and since $n\Lambda$ is integral, 
each non-zero coefficient of $\Theta$ is at least $\frac{1}{n}$.

Now since $n\Lambda$ is integral and since 
$$
lf^*A -(K_X+\Lambda)\sim_\Q 0,
$$ 
 $(X,\Lambda)$ is log birationally bounded, by Proposition \ref{l-bnd-cy-bir-bnd}. Thus  
 $(X,\Theta)$ is also log birationally bounded as $\Supp \Theta=\Supp \Lambda$. Therefore, $(X,\Theta)$ is log bounded by 
\cite[Theorem 1.6]{HMX2} which implies that $(X,\Delta)$ is log bounded as $\Delta\le \Theta$. 
\end{proof}

\subsection{Boundedness of generators of N\'eron-Severi groups}

To treat Theorem \ref{t-log-bnd-cy-fib} in full generality, we need to discuss 
generators of relative N\'eron-Severi groups. We start with bounding global Picard numbers.

\begin{lem}\label{l-cy-fib-bnd-picard-number}
Let $d,r$ be natural numbers and $\varepsilon$ be a positive real number. 
Assume that Theorems \ref{t-log-bnd-cy-fib}, \ref{t-sing-FT-fib-totalspace}, and \ref{t-bnd-comp-lc-global-cy-fib} 
hold in dimension $d-1$. Then there is a natural number $p$ depending only on $d,r,\varepsilon$ 
satisfying the following.
If $(X,B)\to Z$ is a $(d,r,\varepsilon)$-Fano type fibration (as in \ref{d-FT-fib}), then the Picard number $\rho(X)\le p$.
\end{lem}
\begin{proof}
Replacing $X$ with a $\Q$-factorialisation, we can assume $X$ is $\Q$-factorial. 
Running an MMP on $-K_X$ over $Z$, we find $Y$ so that $-K_Y$ is nef and big over $Z$. 
Replace $Y$ with the ample model of $-K_Y$ over $Z$ so that  $-K_Y$ 
becomes ample over $Z$. Let $K_Y+B_Y$ be the pushdown of $K_X+B$. Then 
$(Y,B_Y)\to Z$ is a $(d,r,\varepsilon)$-Fano type fibration. 
Now applying Lemma \ref{l-bnd-cy-fib-ample-case} to $(Y,B_Y)\to Z$ we deduce that 
$Y$ is bounded. 

By construction, if $D$ is a prime divisor on $X$ contracted over $Y$, then 
$$
a(D,Y,0)\le a(D,X,0)=1.
$$ 
Thus, by  \cite[Proposition 2.5]{HX}, there is a birational morphism $X'\to Y$ from a bounded normal 
projective variety such that the induced map $X\bir X'$ is an isomorphism in codimension one.

We can take a resolution $W\to X'$ such that $W$ is bounded. Then there exist finitely many surjective 
smooth projective morphisms $V_i\to T_i$ between smooth varieties, depending only on $d,r,\varepsilon$,  such that 
$W$ is a fibre of $V_i\to T_i$ over some closed point for some $i$. Since smooth morphisms are locally 
products in the complex topology (here we can assume that the ground field is $\C$),  
$\dim_\R H^2(W,\R)$ is bounded by some number $p$ depending only on $d,r,\varepsilon$. 
In particular, since the N\'eron-Severi group $N^1(W)$ is embedded in $H^2(W,\R)$ as a vector space, we get  
$$
\rho(W)\le \dim_\R H^2(W,\R)\le p.
$$ 
Since $X\bir X'$ is an isomorphism in codimension one and since $W\to X'$ is a morphism,  
the induced map $X\bir W$ does not contract divisors, hence $\rho(X)\le \rho(W)\le p$. 
\end{proof}

\begin{prop}\label{p-cy-fib-bnd-Neron-Severi}
Let $d,r$ be natural numbers, $\varepsilon$ be a positive real number, and $\mathfrak{R}\subset [0,1]$ be a 
finite set of rational numbers. 
Assume that Theorems \ref{t-log-bnd-cy-fib}, \ref{t-sing-FT-fib-totalspace}, and \ref{t-bnd-comp-lc-global-cy-fib} 
hold in dimension $d-1$. Then there is a bounded set $\mathcal{P}$ of couples depending only on 
$d,r,\varepsilon, \mathfrak{R}$ satisfying the following. 
Suppose that 
\begin{itemize}
\item $(X,B)\to Z$ is a $(d,r,\varepsilon)$-Fano type fibration (as in \ref{d-FT-fib}), and that

\item  the coefficients of $B$ are in $\mathfrak{R}$.
\end{itemize}
Then there exist a birational map $X\bir X'$  and a reduced divisor $\Sigma'$ on $X'$ 
such that 
\begin{itemize}
\item $X'$ is a $\Q$-factorial normal projective variety, 

\item $X\bir X'$ is an isomorphism in codimension one, 

\item $(X',\Sigma')$ belongs to $\mathcal{P}$, 

\item $\Supp B'\subseteq \Sigma'$ where $B'$ is the birational transform of $B$, and 

\item the irreducible components of $\Sigma'$ generate $N^1(X'/Z)$.
\end{itemize}
\end{prop}

By $N^1(X'/Z)$ we mean $\Pic(X')\otimes \R$ modulo numerical equivalence over $Z$. 
Note that there is a natural surjective map $N^1(X')\to N^1(X'/Z)$.
We prove some lemmas before giving the proof of the proposition.

\begin{lem}\label{l-cy-fib-bnd-Neron-Severi-Mfs}
Assume that Proposition \ref{p-cy-fib-bnd-Neron-Severi} holds in dimension $\le d-1$. Then the proposition holds 
in dimension $d$ when $X$ is $\Q$-factorial and there is a non-birational extremal contraction 
$h\colon X\to Y/Z$.
\end{lem}
\begin{proof}
First note that if $\dim Y=0$, then $X$ is an $\varepsilon$-lc Fano variety with Picard number one 
and $K_X+B\sim_\Q 0$, hence $X$ belongs to a bounded family by \cite[Theorem 1.4]{B-compl}, hence $(X,B)$ is log 
bounded as the coefficients of $B$ are in $\mathfrak{R}$ which implies the result in this case as $N^1(X/Z)$ is generated by the components of $B$.
We can then assume that $\dim Y>0$.

Let $F$ be a general fibre of $h$ and let $K_F+B_F:=(K_X+B)|_F$. Then 
$(F,B_F)$ is $\varepsilon$-lc, $K_F+B_F\sim_\Q 0$, and $B_F$ is big with coefficients in $\mathfrak{R}$, 
hence $F$ belongs to a bounded family by \cite[Theorem 1.4]{B-compl} which implies that 
$(F,B_F)$ is log bounded. Moreover, by adjunction, we can write 
$$
K_X+B\sim_\Q h^*(K_Y+B_Y+M_Y)
$$
where we consider $(Y,B_Y+M_Y)$ as a generalised pair as in Remark \ref{rem-base-fib-gen-pair} below.
By \cite[Theorem 1.4]{B-sing-fano-fib}, $(Y,B_Y+M_Y)$ is generalised $\delta$-lc for 
some fixed $\delta>0$ which depends only on $d,\varepsilon, \mathfrak{R}$. 

By construction, 
$$
K_Y+B_Y+M_Y\sim_\R g^*L
$$
where $g$ denotes the morphism $Y\to Z$. Moreover, since $X$ is of Fano type over $Z$, $Y$ is also of 
Fano type over $Z$ (cf, the proof of \cite[Lemma 2.12]{B-compl} works in the relative setting).
Then 
$$
(Y,B_Y+M_Y)\to Z
$$ 
is a generalised $(d',r,\delta)$-Fano type fibration for some $d'\le d-1$, 
as in \ref{d-gen-FT-fib}. 
By Lemma \ref{l-from-gen-fib-to-usual-fib}, we can find a boundary
$\Delta_Y$ so that $(Y,\Delta_Y)\to Z$ is a $(d',r,\frac{\delta}{2})$-Fano type fibration.
Therefore, applying Lemma \ref{l-bnd-cy-bnd-klt-compl-induction}, there exist bounded natural numbers 
$n,m$ and a boundary $\Lambda_Y$ such that $(Y,\Lambda_Y)$ is klt and 
$n(K_X+\Lambda_Y)\sim mg^*A$. In particular, $(Y,\Lambda_Y)\to Z$ is a $(d',r',\varepsilon':=\frac{1}{n})$-Fano type 
fibration for some fixed $r'$. Moreover, increasing $m$ by adding to $\Lambda_Y$, we can assume that $m>n$.
Since we are assuming Proposition \ref{p-cy-fib-bnd-Neron-Severi} in dimension $d-1$, 
 applying it to $(Y,\Lambda_Y)\to Z$ we deduce that  there exist a birational map 
$Y\bir Y'/Z$ to a $\Q$-factorial normal projective variety and 
a reduced divisor $\Sigma_{Y'}$ on $Y'$ satisfying the properties listed in \ref{p-cy-fib-bnd-Neron-Severi}. 

By Lemma \ref{l-contraction-after-flops}, there exists a birational map $X\bir X'/Z$ which is an isomorphism in codimension one 
so that the induced map $X'\bir Y'$ is an extremal contraction (hence a morphism) and 
$X'$ is normal projective and $\Q$-factorial. Let $B'$ on $X'$ be the birational transform of $B$ and let 
$\Lambda_{Y'}$ on $Y'$ be the birational transform of $\Lambda_Y$.  
To ease notation we can replace $(X,B)$ and $(Y,\Lambda_Y)$ with $(X',B')$ and $(Y',\Lambda_Y')$ 
and denote $\Sigma_{Y'}$ by $\Sigma_Y$. By construction, $\Supp \Lambda_Y\le \Sigma_Y$.

Since $\Supp \Lambda_Y\subseteq \Sigma_Y$ and since $(Y,\Sigma_Y)$ is log bounded, 
there is a very ample divisor $H$ on $Y$ with bounded $s:=H^{\dim Y}$ 
such that 
$$
H-(K_Y+\Lambda_Y)\sim_\Q H-\frac{m}{n}g^*A
$$ 
is ample which implies that $H-g^*A$ is ample as $m>n$. Then  
$$
H-g^*L=H-g^*A+g^*(A-L)
$$
is ample as $A-L$ is ample. Thus $(X,B)\to Y$ is a $(d,s,\varepsilon)$-Fano type fibration in view of $K_X+B\sim_\R h^*g^*L$. 

Replacing $H$ we can in addition assume that $H-\Sigma_Y$ is ample. Thus we can find $0\le P\sim_\R h^*H$ such that 
$P\ge h^*\Sigma_Y$. Now applying 
Proposition \ref{l-bnd-cy-bnd-lct-special} to $(X,B)\to Y$, there is a fixed rational number $t\in (0,1)$ 
such that $(X,B+2tP)$ is klt. Thus $(X,B+2th^*\Sigma_Y)$ is klt, so
$$
(X,\Theta:=B+th^*\Sigma_Y)
$$ 
is $\frac{\varepsilon}{2}$-lc. Therefore, from 
$$
K_X+\Theta=K_X+B+th^*\Sigma_Y\sim_\R h^*(g^*L+t\Sigma_Y)
$$
we deduce that $(X,\Theta)\to Y$ is a $(d,s,\frac{\varepsilon}{2})$-Fano type fibration, 
perhaps after replacing $H$ with $2H$ and replacing $s$ accordingly.

On the other hand, the coefficients of $th^*\Sigma_Y$ 
belong to a fixed finite set because $t$ is fixed, the Cartier index of $\Sigma_Y$ is bounded by Lemma \ref{l-bnd-Cartier-index-family} applied to $(Y,0)$ and $(Y,\frac{1}{2}\Sigma)$, 
and the coefficients of  $th^*\Sigma_Y$ are less than $1$. Thus the coefficients of $\Theta$ belong to a 
fixed finite set. Moreover, since $X\to Y$ is extremal and $\Theta$ is big over $Y$, $\Theta$ is ample 
over $Y$, hence  
$$
-(K_X+\frac{1}{2}\Theta)\sim_\R \frac{1}{2}\Theta/Y
$$ 
is ample over $Y$. 
Therefore, applying Lemma \ref{l-bnd-cy-fib-ample-case} to $(X,\Theta)\to Y$ (by taking $\Delta=\frac{1}{2}\Theta$) 
we deduce that $(X,\frac{1}{2}\Theta)$ is log bounded.
Since $B$ is big over $Y$, it is ample over $Y$, 
hence the components of $B$ and $h^*\Sigma_Y$ together  generate $N^1(X/Z)$ 
as the components of $\Sigma_Y$ generate $N^1(Y/Z)$. Now let $\Sigma:=\Supp \Theta$.
\end{proof}

\begin{lem}\label{l-cy-fib-bnd-Neron-Severi-sqf}
Proposition \ref{p-cy-fib-bnd-Neron-Severi} holds when $X\to Z$ is a small $\Q$-factorialisation.
\end{lem}
\begin{proof}
We will apply induction on the relative Picard number $\rho(X/Z):=\dim_\R N^1(X/Z)$. 
By Lemma \ref{l-cy-fib-bnd-picard-number}, $\rho(X)$ is bounded, so 
$\rho(X/Z)$ is bounded as well because $\rho(X/Z)\le \rho(X)$. 
The case $\rho(X/Z)=0$ is trivial in which case $X\to Z$ is an isomorphism and 
$(X,B)$ is log bounded, so we assume $\rho(X/Z)>0$.

By Lemma \ref{l-bnd-cy-bnd-klt-compl-induction}, our assumptions imply Theorem \ref{t-bnd-comp-lc-global-cy-fib} 
in dimension $d$. Since $X\to Z$ is birational, $-(K_X+B)$ is big over $Z$, hence applying 
the theorem there exist bounded natural numbers $n,m$ and a boundary $\Lambda\ge B$ such that 
$(X,\Lambda)$ is klt and $n(K_X+\Lambda)\sim mf^*A$. In particular, $n(K_X+\Lambda)$ is Cartier and 
$(X,\Lambda)$ is $\frac{1}{n}$-lc. Replacing $B$ with $\Lambda$, $\epsilon$ with $\frac{1}{n}$, 
$A$ with $2mA$, and replacing $r,\mathfrak{R}$ accordingly, we can assume that $n(K_X+B)$ is Cartier for 
some fixed natural number $n$. Replacing $n$ with $2n$ we can assume $n\ge 2$.

Let $B_Z$ be the pushdown of $B$. By boundedness of length of extremal rays \cite{kawamata-bnd-ext-ray}, 
$$
K_Z+B_Z+(2d+1)A
$$ 
is ample. Thus taking a general member $G\in |n(2d+1)f^*A|$, adding $\frac{1}{n}G$ to $B$, 
and then replacing $A$ with $(2d+2)A$ (to keep the ampleness of $A-L$), we can assume that 
$K_X+B$ is the pullback of some ample divisor on $Z$ and that $B-\frac{1}{2}f^*A$ is pseudo-effective. 
We have used the assumption $n\ge 2$ to make sure that the $\epsilon$-lc property of $(X,B)$ is preserved.

By the cone theorem \cite[Theorem 3.7]{kollar-mori}, we can decompose 
$X\to Z$ into a sequence 
$$
X=X_1\to X_2 \to \cdots \to X_l=Z
$$
of extremal contractions. 
Let $B_i$ be the pushdown of $B$. Then 
$$
(K_X+B)^d=\vol(K_X+B)=\vol(K_{X_i}+B_i)\le \vol(A)=A^d\le r,
$$ 
hence there are only finitely many possibilities for $\vol(K_{X_i}+B_i)$ as $n(K_X+B)$ is Cartier. 
Therefore, by \cite[Theorem 6]{DST}, the set of such $(X_i,B_i)$ is log bounded.
In particular, there is a very ample divisor $G_{l-1}$ on $X_{l-1}$ 
with bounded $G_{l-1}^d$ such that $G_{l-1}-A_{l-1}$ is ample where $A_{l-1}$ is the 
pullback of $A$ (here we are using the property that $B-\frac{1}{2}f^*A$ is pseudo-effective). 

Let $G$ be the pullback of $G_{l-1}$ to $X$. Let $\Theta=B+\frac{1}{n}P$ for some general element $P\in |nG|$.  
Then $K_X+\Theta\sim_\Q 0/X_{l-1}$ and 
$$  
\begin{array}{l l}
2G-(K_X+\Theta) &=2G-(K_X+B)-\frac{1}{n}P\\
& \sim_\Q G-(K_X+B)\\
& =G-f^*A+f^*A-f^*L
\end{array}
$$
is the pullback of an ample divisor on $X_{l-1}$ as $G_{l-1}-A_{l-1}$ and $A-L$ are ample.
Thus $(X,\Theta)\to X_{l-1}$ is a $(d,u,\varepsilon)$-Fano type fibration for some fixed number $u$. 

Now $\rho(X/X_{l-1})<\rho(X/Z)$.
Therefore, by induction on the relative Picard number, there is a birational map $X\bir X'/X_{l-1}$ and a reduced divisor 
 $\Sigma'$ on $X'$ satisfying the properties listed in \ref{p-cy-fib-bnd-Neron-Severi} with $\Theta, X_{l-1}$ instead of $B,Z$.
Now since $P'$, the birational transform of $P$, is the pullback of some ample$/Z$ divisor on $X_{l-1}$, since $P'\le \Sigma'$, 
and since $X_{l-1}\to X_l=Z$ is extremal, the components of $\Sigma'$ generate $N^1(X'/Z)$. This proves the lemma.
\end{proof}

\begin{proof}[Proof of Proposition \ref{p-cy-fib-bnd-Neron-Severi}]
Replacing $X$ with a $\Q$-factorialisation, we can assume $X$ is $\Q$-factorial. 
By Lemma \ref{l-cy-fib-bnd-picard-number},  the Picard number $\rho(X)$ is bounded. 
We will apply induction on dimension and induction on the relative Picard number $\rho(X/Z)$, in the 
$\Q$-factorial case. We will assume that  $X\to Z$ is not an isomorphism otherwise the proposition holds by taking 
$X'=X$ and $\Sigma'=\Supp B$ as in this case $(X,B)$ would be log bounded by definition of $(d,r,\varepsilon)$-Fano type fibrations.

First we prove the proposition assuming that there is a birational map $h\colon X\bir Y/Z$ 
to a normal projective variety such that $h^{-1}$ does not contract any divisor but $h$ contracts some divisor.
Since $K_X+B\sim_\R 0/Z$, $(Y,B_Y)$ is klt where $B_Y$ is the pushdown of $B$. 
Replacing $Y$ with a $\Q$-factorialisation we can assume it is $\Q$-factorial. 
The log discrepancy of any prime divisor $D$ contracted by $h$ satisfies 
$$
a(D,Y,B_Y)=a(D,X,B)\le 1.
$$ 
Thus modifying $Y$ by extracting all such divisors except one, we can assume that $h$ contracts a 
single prime divisor $D$. Moreover, replacing $h$ with the extraction morphism determined by $D$, 
we can assume that $h$ is an extremal divisorial contraction. 

Now $(Y,B_Y)\to Z$ is a $(d,r,\varepsilon)$-Fano type fibration and $\rho(Y/Z)<\rho(X/Z)$, 
so by the induction hypothesis, there exist a birational map 
$Y\bir Y'/Z$ to a normal projective variety and 
a reduced divisor $\Sigma_{Y'}$ on $Y'$ satisfying the properties of the proposition.  
Replacing $Y$ with $Y'$ and replacing $X$ accordingly (as in the previous paragraph) 
we can assume $Y=Y'$. We change the notation $\Sigma_{Y'}$ to $\Sigma_Y$.

Since $\Supp B_Y\subseteq \Sigma_Y$ and since $(Y,\Sigma_Y)$ is log bounded, by 
\cite[Theorem 1.8]{B-BAB} (=Theorem \ref{t-bnd-lct}), 
there is a fixed rational number $t>0$ such that 
$$
(Y,\Theta_Y:=B_Y+t\Sigma_Y)
$$ 
is $\frac{\varepsilon}{2}$-lc. 
Thus since $a(D,Y,\Theta_Y)\le 1$, applying \cite[Proposition 2.5]{HX} we deduce that 
there is a birational contraction $X'\to Y$ extracting $D$ but no other divisors and such that if 
$K_{X'}+\Theta'$ is the pullback of $K_Y+\Theta_Y$, then $(X',\Theta')$ is log bounded. 
In addition, from the proof of \cite[Proposition 2.5]{HX}  we can see that 
if 
$$
\Sigma'=\Supp (D+\Theta'),
$$ 
then $(X',\Sigma')$ is log bounded.
Now since $Y$ is $\Q$-factorial, $X'=X$. For convenience we change the notation $\Theta',\Sigma'$ to $\Theta,\Sigma$. 
Since $K_X+B\sim_\Q 0/Y$ and $B_Y\le \Theta_Y$, we have $B\le \Theta$, hence 
$$
\Supp B\subseteq \Supp \Theta\subseteq \Sigma.
$$ 
By construction, $\Sigma$ generates $N^1(X/Z)$, 
so we are done in this case.

Now we prove the proposition in general. 
If $X\to Z$ is not birational, then running an MMP on $K_X$ ends with a Mori fibre space 
$\tilde{X}\to Y/Z$; applying the above we can assume that $X\bir \tilde{X}$ does not contract any divisor, 
hence replacing $X$ we can assume $X=\tilde{X}$; we can then apply Lemma \ref{l-cy-fib-bnd-Neron-Severi-Mfs}. 
Now assume that $X\to Z$ is birational. Applying the above again reduces the proposition to the case 
when $X\to Z$ is a small contraction. But then we can apply Lemma \ref{l-cy-fib-bnd-Neron-Severi-sqf}.
\end{proof}

\subsection{Boundedness of Fano type fibrations}

In this subsection, we treat Theorems \ref{t-bnd-cy-fib} and \ref{t-log-bnd-cy-fib} inductively. First we prove an auxilliary lemma. 

\begin{lem}\label{l-dominating-Sigma}
Let $d,r$ be natural numbers, $\varepsilon$ be a positive real number, and $\mathcal{P}$ be a bounded set of couples. Assume that Theorems \ref{t-log-bnd-cy-fib}, \ref{t-sing-FT-fib-totalspace}, and \ref{t-bnd-comp-lc-global-cy-fib} 
hold in dimension $d-1$.
Then there exist natural numbers $m,l$ satisfying the following. Let $(X,B)\to Z$ be a $(d,r,\varepsilon)$-Fano type fibration where $K_X$ is $\Q$-Cartier, and let $\Sigma$ be a reduced divisor on $X$ so that $(X,\Sigma)\in \mathcal{P}$ and $\Sigma-f^*A$ is pseudo-effective. Then 
$$
|m(lf^*A-K_X)-\Sigma|
$$ 
defines a birational map.   
\end{lem}
\begin{proof}
Since $(X,\Sigma)$ belongs to a bounded family, we can find a general very ample divisor $H$ so that $H^d$ is bounded and $H-\Sigma\sim D$ for some $D\ge 0$. Replacing $\Sigma$ with $H$ and replacing $\mathcal{P}$ accordingly we can  
assume that $\Sigma$ is very ample. 

Let $X'$ be the ample model of $-K_X$ over $Z$ and let $X''$ be a $\Q$-factorialisation of $X'$. 
Since $\Sigma$ is nef, the pullback of $\Sigma''$ is bigger than or equal to the pullback of $\Sigma$ on any common resolution of $X,X''$, by the negativity lemma. On the other hand, the pullback of $lf''^*A-K_{X''}$ is smaller than or equal to the pullback of $lf^*A-K_{X}$ on any common resolution of $X,X''$ for any $l$ where $f''$ denotes $X''\to Z$.
Thus if we find $m,l$ so that 
$$
|m(lf''^*A-K_{X''})-\Sigma''|
$$ 
defines a birational map, then 
$$
|m(lf^*A-K_{X})-\Sigma|
$$ 
also defines a birational map. It is then enough to find $m,l$ so that 
$$
|m(lf'^*A-K_{X'})-\Sigma'|
$$ 
defines a birational map where $f'$ denotes $X'\to Z$. 

By Proposition \ref{l-bnd-cy-fib-ample-case} and its proof, $X'$ is bounded and  
we can find fixed $m,l$ so that $m(lf'^*A-K_{X'})$ is very ample. 
Since $\Sigma-f^*A$ is pseudo-effective and $(X,\Sigma)$ is bounded, 
$$
\vol(m(lf^*A-K_X)+\Sigma)
$$ 
is bounded, hence if $\alpha \colon V\to X$ and $\beta\colon V\to X'$ is a common resolution, then 
$$
\vol(\beta^*m(lf'^*A-K_{X'})+\alpha^* \Sigma)
$$
is bounded. Thus 
$$
(\beta^*m(lf'^*A-K_{X'}))^{d-1}\cdot \alpha^*\Sigma
$$ 
is bounded. This implies that 
$$
(m(lf'^*A-K_{X'}))^{d-1}\cdot \Sigma'
$$ 
is bounded which in turn implies that $(X',\Sigma')$ belongs to a bounded family. 
Therefore, relpacing $m$ with a bounded multiple we can assume that 
$$
|m(lf'^*A-K_{X'})-\Sigma'|
$$ 
defines a birational map as desired.
\end{proof}

\begin{lem}\label{l-log-bnd-cy-fib-induction}
Assume that Theorems \ref{t-log-bnd-cy-fib}, \ref{t-sing-FT-fib-totalspace}, and \ref{t-bnd-comp-lc-global-cy-fib} 
hold in dimension $d-1$. Then Theorems \ref{t-bnd-cy-fib} and \ref{t-log-bnd-cy-fib} hold in dimension $d$.
\end{lem}
\begin{proof}
It is enough to treat \ref{t-log-bnd-cy-fib} as it implies \ref{t-bnd-cy-fib} by taking $\Delta=0$. 
Let $(X,B)\to Z$ be a $(d,r,\varepsilon)$-Fano type fibration and $0\le \Delta \le B$, as in \ref{t-log-bnd-cy-fib}. 
If $\dim Z=0$, then $X$ belongs to a bounded family by \cite[Corollary 1.4]{B-BAB} from which 
we can deduce that $(X,\Delta)$ is log bounded. We can then assume that $\dim Z>0$. 
Changing the coefficients of $\Delta$ we can assume that all its coefficients are equal to a fixed rational number, 
and that 
$$
\Supp (B-\Delta)=\Supp B.
$$ 
This in particular implies that 
$$
-(K_X+\Delta)\sim_\R B-\Delta/Z
$$ 
is big over $Z$. 

Now by Lemma \ref{l-bnd-cy-bnd-klt-compl-induction}, 
our assumptions imply Theorem \ref{t-bnd-comp-lc-global-cy-fib} in dimension $d$, hence applying the theorem 
we can find bounded natural numbers $n,m\ge 2$ and a boundary $\Lambda\ge \Delta$ such that $(X,\Lambda)$ is klt 
and $n(K_X+\Lambda)\sim mf^*A$. Replacing $B$ with $\Lambda$, $\varepsilon$ with $\frac{1}{n}$, and $A$ with $2mA$ 
we can assume that the coefficients of $B$ belong to some fixed finite set of rational numbers and that $K_X+B$ is nef. 

By Proposition \ref{p-cy-fib-bnd-Neron-Severi}, 
there exist a birational map $X\bir X'/Z$  and a reduced divisor $\Sigma'$ on $X'$ satisfying the properties 
listed in the proposition. Since $(X',\Sigma')$ is log bounded,  
there is a very ample divisor $G'$ on $X'$ with bounded $G'^d$ and bounded 
$G'^{d-1}\cdot \Sigma'$. Moreover, as $\Supp B'\subseteq \Sigma'$, 
$(X',B')$ is log bounded. Since $G'^{d-1}\cdot (K_{X'}+B')$ is bounded,  
$G'^{d-1}\cdot f'^*A$ is bounded, hence we can assume that $\Sigma'\ge P'$ for some general member $P'$ of $|f'^*A|$ where $f'$ denotes $X'\to Z$.

Now by Lemma \ref{l-dominating-Sigma} applied to $(X',B'),\Sigma'\to Z$, there exist bounded natural numbers $p,q$ such that 
$$
|q(pf'^*A-K_{X'})-\Sigma'|
$$ 
defines a birational map. Using a member $D'$ of this linear system we find $\Pi'=\frac{1}{q}(\Sigma'+D')$ so that   
$$
q(K_{X'}+\Pi')\sim qpf'^*A. 
$$
The pair 
$
(X',\Pi')
$ 
is log bounded because $G'^{d-1}\cdot \Pi'$ is bounded.

Replacing $G'$ with a multiple we can assume that $G'-B'$ and $G'-\Pi'$ 
are big. 
Therefore, by \cite[Theorem 1.8]{B-BAB} (=Theorem \ref{t-bnd-lct}), there is a 
fixed rational number $t\in (0,1)$ such that 
$$
(X',B'+t\Pi')
$$ 
is $\frac{\varepsilon}{2}$-lc, hence 
$$
(X',C':=(1-t)B'+t\Pi')
$$ 
is $\frac{\varepsilon}{2}$-lc. Note that 
$$
K_{X'}+C'=(1-t)(K_{X'}+B')+t(K_{X'}+\Pi')\sim_\Q (1-t)f'^*L+tpf'^*A
$$
and that $C'\ge (1-t)\Delta'$.
Replacing $B'$ with $C'$ hence replacing $B$ with the birational transform of $C'$,
replacing $\Delta$ with $(1-t)\Delta$, replacing $A$ accordingly, replacing $\Sigma'$ with $\Supp \Pi'$, 
and replacing $\varepsilon$ with $\frac{\varepsilon}{2}$, we can assume that $\Supp B'=\Sigma'$.
In addition, by the above, we can assume that $P'\le \Sigma'$ where $P$ is some member of $|f^*A|$. 

Let $H$ be an ample $\Q$-divisor on $X$ and let $H'$ be its birational 
transform on $X'$. Since the components of $\Sigma'$ generate $N^1(X'/Z)$, there exists an $\R$-divisor 
$R'\equiv H'/Z$ such that $\Supp R'\subseteq \Sigma'$. In particular, if $R$ is the birational transform of $R'$ 
on $X$, then $R$ is ample over $Z$. Replacing $R'$ with a small multiple and adding a multiple of 
$P'$ to it, we can assume that  $R$ is globally 
ample. Since $\Supp R\subseteq \Supp B$, rescaling $R$ we can in addition assume that 
$$
\Theta:=B+R\ge \frac{1}{2}\Delta,
$$ 
that 
the coefficients of $\Theta$ are $\ge \frac{\delta}{2}$, and that $(X,\Theta)$ is $\frac{\varepsilon}{2}$-lc. 

By construction, 
$$
\Supp \Theta'=\Supp B'=\Sigma'
$$ 
where $\Theta'$ is the birational transform of $\Theta$. 
Thus $(X,\Theta)$ is log birationally bounded. 
Moreover, $K_X+\Theta$ is ample as $K_X+B$ is nef and $R$ is ample. Therefore, applying \cite[Theorem 1.6]{HMX2}, 
we deduce that $(X,\Theta)$ is log bounded which in particular means that $(X,\Delta)$ is log bounded.
\end{proof}

\subsection{Lower bound on lc thresholds}

\begin{lem}\label{l-bnd-cy-bnd-lct-usual}
Assume that Theorems \ref{t-log-bnd-cy-fib}, \ref{t-sing-FT-fib-totalspace}, and \ref{t-bnd-comp-lc-global-cy-fib} 
hold in dimension $d-1$. Then Theorem \ref{t-sing-FT-fib-totalspace} holds in dimension $d$. 
\end{lem}
\begin{proof}
Assume that $(X,B)\to Z$ is a $(d,r,\varepsilon)$-Fano type fibration and $P\ge 0$ is $\R$-Cartier 
such that either $f^*A+B-P$ or $f^*A-K_X-P$ is pseudo-effective. 
Taking a $\Q$-factorialisation we can assume $X$ is $\Q$-factorial.  
First assume that  $f^*A+B-P$ is pseudo-effective.
Since $A-L$ is nef, $f^*A-(K_X+B)$ is nef, hence $2f^*A-K_X-P$ is pseudo-effective. Thus replacing 
$A$ with $2A$, it is enough to treat the theorem 
in the case when  $f^*A-K_X-P$ is pseudo-effective. 

By Lemma \ref{l-log-bnd-cy-fib-induction}, Theorem \ref{t-log-bnd-cy-fib} holds in dimension $d$. 
Let $D\in |f^*A|$ be a general element. Let $\Theta:=B+\frac{1}{2}D$. Then 
$$
K_X+\Theta\sim_\R f^*(L+\frac{1}{2}A)
$$
and $(X,\Theta)$ is $\varepsilon'$-lc where $\varepsilon'=\min\{\varepsilon,\frac{1}{2}\}$. 
Thus $(X,\Theta)\to Z$ is a $(d,2^{d}r,\varepsilon')$-Fano type fibration. 
Applying \ref{t-log-bnd-cy-fib}, we  
deduce that  $(X,D)$ is log bounded. 
Thus there is a very ample divisor $H$ on $X$ such that 
$H^d$ is bounded and $H+K_X-D$ is ample. 

Since $f^*A-(K_X+B)$ is nef, 
$$
H-B\sim H+K_X-D+f^*A-(K_X+B)
$$ 
is ample. On the other hand, since $Q:=f^*A-K_X-P$ is pseudo-effective, 
$$
H-P=H+K_X-f^*A+Q
$$
is big, hence $|H-P|_\R\neq \emptyset$. Now by [\ref{B-BAB}, Theorem 1.8] (=Theorem \ref{t-bnd-lct}), 
there is a real number $t>0$ 
depending only on $d,H^d,\varepsilon$ such that $(X,B+tP)$ is klt. By construction, $t$ depends only on 
$d,r,\varepsilon$.
\end{proof}

\subsection{Proofs of \ref{t-bnd-cy-fib}, \ref{t-log-bnd-cy-fib}, 
\ref{t-sing-FT-fib-totalspace}, \ref{t-bnd-comp-lc-global-cy-fib}}

We are now ready to prove several of the main results of this paper. We apply induction so we assume that 
 \ref{t-log-bnd-cy-fib}, \ref{t-sing-FT-fib-totalspace}, and \ref{t-bnd-comp-lc-global-cy-fib} hold in dimension $d-1$.

\begin{proof}[Proof of Theorems \ref{t-bnd-cy-fib} and \ref{t-log-bnd-cy-fib}]
Theorems \ref{t-bnd-cy-fib} and \ref{t-log-bnd-cy-fib} follow from Theorems \ref{t-log-bnd-cy-fib}, \ref{t-sing-FT-fib-totalspace}, and 
\ref{t-bnd-comp-lc-global-cy-fib} in dimension $d-1$, and Lemma \ref{l-log-bnd-cy-fib-induction}.
\end{proof}

\begin{proof}[Proof of Theorem \ref{t-sing-FT-fib-totalspace}]
This follows from Theorems \ref{t-log-bnd-cy-fib}, \ref{t-sing-FT-fib-totalspace}, 
and \ref{t-bnd-comp-lc-global-cy-fib} in dimension $d-1$, and Lemma \ref{l-bnd-cy-bnd-lct-usual}.
\end{proof}

\begin{proof}[Proof of Theorem \ref{t-bnd-comp-lc-global-cy-fib}]
This follows from Theorems \ref{t-log-bnd-cy-fib}, \ref{t-sing-FT-fib-totalspace}, and 
\ref{t-bnd-comp-lc-global-cy-fib} in dimension $d-1$, and Lemma \ref{l-bnd-cy-bnd-klt-compl-induction}.
\end{proof}

\section{\bf Generalised Fano type fibrations}

In this section, we discuss singularities and boundedness of Fano type log Calabi-Yau fibrations in the 
context of generalised pairs. 

\subsection{Adjunction for generalised fibrations.}\label{fib-adj-setup}
Consider the following set-up. Assume that 
\begin{itemize}
\item $(X,B+M)$ is a generalised sub-pair with data $X'\to X$ and $M'$,

\item $f\colon X\to Z$ is a contraction with $\dim Z>0$, 

\item $(X,B+M)$ is generalised sub-lc over the generic point of $Z$, and 

\item $K_{X}+B+M\sim_\R 0/Z$.
\end{itemize}
We define the discriminant divisor $B_Z$ for the above setting, similar to the definition in the introduction. 
Let $D$ be a prime divisor on $Z$. Let $t$ be the generalised lc threshold of $f^*D$ with respect to $(X,B+M)$ 
over the generic point of $D$. This makes sense even if $D$ is not $\Q$-Cartier because we only need 
the pullback $f^*D$ over the generic point of $D$ where $Z$ is smooth. 
We then put the coefficient of  $D$ in $B_Z$ to be $1-t$. Note that since $(X,B+M)$ is generalised 
sub-lc over the generic point of $Z$,  $t$ is a real number, that is, it is not $-\infty$ or $+\infty$.
Having defined $B_Z$, we can find $M_Z$ giving 
$$
K_{X}+B+M\sim_\R f^*(K_Z+B_Z+M_Z)
$$
where $M_Z$ is determined up to $\R$-linear equivalence. 
We call $B_Z$ the \emph{discriminant divisor of adjunction} for $(X,B+M)$ over $Z$. 

Let $Z'\to Z$ be a birational contraction from a normal variety. There is 
a birational contraction $X'\to X$ from a normal variety so that the induced map $X'\bir Z'$ is a morphism.  
Let $K_{X'}+B'+M'$ be the 
pullback of $K_{X}+B+M$. We can similarly define $B_{Z'},M_{Z'}$ for $(X',B'+M')$ over $Z'$. In this way we 
get the \emph{discriminant b-divisor ${\bf{B}}_Z$ of adjunction} for $(X,B+M)$ over $Z$. 
Fixing a choice of $M_Z$ we can pick the $M_{Z'}$ consistently so that it also defines a b-divisor ${\bf{M}}_Z$ 
which we refer to as the \emph{moduli b-divisor of adjunction} for $(X,B+M)$ over $Z$. 

\begin{rem}\label{rem-base-fib-gen-pair}
\emph{
Assume that $M=0$, $B$ is a $\Q$-divisor, $(X,B)$ is projective, and that $(X,B)$ is lc over 
the generic point of $Z$. Then ${\bf{M}}_Z$ is b-nef b-$\Q$-Cartier, that is, we can pick $Z'$ so that 
$M_{Z'}$ is a nef $\Q$-divisor and for any resolution $Z''\to Z'$, $M_{Z''}$ is the pullback of $M_{Z'}$ 
[\ref{B-compl}, Theorem 3.6] (this is derived from [\ref{FG-lc-trivial}] which is in turn derived from [\ref{ambro-lc-trivial}]
and this in turn is based on [\ref{kaw-subadjuntion}]). We can then consider 
$(Z,B_Z+M_Z)$ as a generalised pair with nef part $M_{Z'}$.
When $M\neq 0$, the situation is more complicated, see [\ref{Filipazzi}] for recent 
advances in this direction which we will not use in this paper. }
\end{rem}

\smallskip

\subsection{Boundedness of generalised Fano type fibrations}

\begin{proof}[Proof of Theorem \ref{t-log-bnd-cy-gen-fib}]
Since $\Delta\le B$, 
$$
-(K_X+\Delta)\sim_\R -(K_X+B)+B-\Delta \sim_\R B-\Delta/Z 
$$
is pseudo-effective over $Z$. And since $-K_X$ is big over $Z$, 
$$
-(K_X+\frac{1}{2}\Delta)=-\frac{1}{2}(K_X+\Delta)-\frac{1}{2}K_X
$$
is big over $Z$. Thus replacing $\Delta$ with $\frac{1}{2}\Delta$ and replacing $\tau$ with $\frac{1}{2}\tau$, we can assume that $-(K_X+\Delta)$ is big over $Z$. 
Now apply Lemma \ref{l-from-gen-fib-to-usual-fib} and Theorem \ref{t-log-bnd-cy-fib}.
\end{proof}

\subsection{Lower bound for lc thresholds: proof of \ref{t-sing-gen-FT-fib-totalspace}}

\begin{proof}[Proof of Theorem \ref{t-sing-gen-FT-fib-totalspace}]
\emph{Step 1.}
\emph{In this step, we make some preparations.}
Let $(X,B+M)\to Z$ and $P$ be as in Theorem \ref{t-sing-gen-FT-fib-totalspace} in dimension $d$.
Taking a $\Q$-factorialisation we can assume $X$ is $\Q$-factorial.  
Assume that  $f^*A+B+M-P$ is pseudo-effective.
Since 
$$
f^*A-(K_X+B+M)\sim_\R f^*(A-L)
$$ 
is nef, $2f^*A-K_X-P$ is pseudo-effective. Thus replacing 
$A$ with $2A$, it is enough to treat Theorem \ref{t-sing-gen-FT-fib-totalspace} 
in the case when  $f^*A-K_X-P$ is pseudo-effective.\\

\emph{Step 2.}
\emph{In this step, we take a log resolution and introduce some notation.}
Since $B$ is effective and $M$ is pseudo-effective (as it is the pushdown of a nef divisor),
$B+M$ is pseudo-effective. Moreover, since  $-K_X$ is big over $Z$, 
$B+M$ is big over $Z$, hence  $B+M+f^*A$ is big globally by Lemma \ref{l-div-pef-on-FT}.
Let $\phi\colon X'\to X$ be a log resolution of $(X,B)$ 
on which the nef part $M'$ of $(X,B+M)$ resides. Write 
$$
K_{X'}+B'+M'=\phi^*(K_X+B+M).
$$
Since $(X,B+M)$ is generalised $\varepsilon$-lc, the coefficients of $B'$ do not exceed $1-\varepsilon$.
We can write 
$$
\phi^*(B+M+f^*A)\sim_\R G'+H'
$$ 
where $G'\ge 0$ and $H'$ is general ample. Replacing $\phi$ we can assume $\phi$ is a log resolution of 
$(X,B+P+\phi_*G')$.

Pick a small $\alpha>0$ and pick a general 
$$
0\le R'\sim_\R \alpha H'+(1-\alpha)M'.
$$
Since $M'$ is nef, $\phi^*M=M'+E'$ where $E'$ is effective and exceptional. 
Let 
$$
\Delta':=B'-\alpha \phi^*B-\alpha E'+\alpha G'+R'.
$$
We can make the above choices so that the coefficients of $\Delta'$ do not exceed $1-\frac{\varepsilon}{2}$ 
and so that $(X',\Delta')$ is log smooth.\\
 
\emph{Step 3.} 
\emph{In this step, we show that $(X,\Delta)\to Z$ is a $(d,r,\frac{\varepsilon}{2})$-Fano type fibration 
where $\Delta=\phi_*\Delta'$.}
By construction, we have 
$$
 \begin{array}{l l}
K_{X'}+\Delta' & =K_{X'}+ B'-\alpha \phi^*B-\alpha E'+\alpha G'+R'\\
& \sim_\R K_{X'}+ B'-\alpha \phi^*B-\alpha E'+\alpha G'+\alpha H'+(1-\alpha)M'\\
& \sim_\R K_{X'}+ B'-\alpha \phi^*B-\alpha E'+\alpha \phi^*(B+M+f^*A)+(1-\alpha)M'\\
& \sim_\R K_{X'}+ B'-\alpha \phi^*B-\alpha \phi^* M+\alpha \phi^*(B+M+f^*A)+M'\\
& \sim_\R K_{X'}+ B'+M'+\alpha\phi^*f^*A\sim_\R \phi^*f^*(L+\alpha A).
\end{array}
$$
Therefore,  
$$
K_X+\Delta\sim_\R f^*(L+\alpha A).
$$ 
Choosing $\alpha$ small enough we can ensure $A-(\alpha A+L)$ is 
ample. On the other hand,  since $K_{X'}+\Delta'\sim_\Q 0/X$,  we have $K_{X'}+\Delta'=\phi^*(K_X+\Delta)$, 
hence $(X,\Delta)$ is $\frac{\varepsilon}{2}$-lc because 
the coefficients of $\Delta'$ do not exceed $1-\frac{\varepsilon}{2}$.
Thus $(X,\Delta)\to Z$ is a $(d,r,\frac{\varepsilon}{2})$-Fano type fibration.\\

\emph{Step 4.}
\emph{In this step, we finish the proof.}
By Theorem \ref{t-sing-FT-fib-totalspace}, there is a real number $t>0$ 
depending only on $d,r,\varepsilon$ such that $(X,\Delta+2tP)$ is klt.
Then letting $P'=\phi^* P$ we see that the coefficients of $\Delta'+2tP'$ do not exceed $1$ as 
$$
K_{X'}+\Delta'+2tP'=\phi^*(K_X+\Delta+2tP).
$$
 Thus 
the coefficients of 
$$
B'-\alpha \phi^*B-\alpha E'+2tP'
$$ 
do not exceed $1$.
Now $t$ is independent of the choice of $\alpha$, so taking the limit as $\alpha$ approaches zero, we see that the 
coefficients of $B'+2tP'$ do not exceed $1$. Therefore, the coefficients of $B'+tP'$ are 
strictly less than $1$ because the coefficients of $B'$ do not exceed $1-\varepsilon$, hence $(X,B+tP+M)$ is generalised klt as 
\[K_{X'}+B'+tP'+M'=\phi^*(K_X+B+tP+M).\qedhere\]
\end{proof}

\subsection{Discriminant divisors of Fano type fibrations: proofs of \ref{t-sh-conj-bnd-base} and \ref{t-sh-conj-bnd-base-gen-fib}}

\begin{proof}[Proof of Theorems \ref{t-sh-conj-bnd-base} and \ref{t-sh-conj-bnd-base-gen-fib}]
Since \ref{t-sh-conj-bnd-base} is a special case of \ref{t-sh-conj-bnd-base-gen-fib} we treat the latter only.
By induction we can assume that Theorem \ref{t-sh-conj-bnd-base-gen-fib} holds in dimension $d-1$.
Let $(X,B+M)\to Z$ be a generalised $(d,r,\varepsilon)$-Fano type fibration.
Let $D$ be  a prime divisor over $Z$. First assume that the centre of $D$ on $Z$ is positive-dimensional. 
Take a resolution $Z'\to Z$ so that $D$ is a divisor on $Z'$. Take a log resolution $\phi\colon X'\to X$ of $(X,B)$ so that 
the nef part $M'$ of $(X,B+M)$  is on $X'$ and that the induced map $f'\colon X'\bir Z'$ is a morphism. 
Replacing $X'$ we can assume that $\phi$ is a log resolution of $(X,B+\phi_*f'^*D)$.

Let $K_{X'}+B'+M'$ be the pullback of $K_X+B+M$.
Let $t$ be the generalised lc threshold of $f'^*D$ with respect to $(X',B'+M')$ over the 
generic point of $D$: this coincides with the lc threshold of $f'^*D$ with respect to $(X',B')$ over the 
generic point of $D$ because $M'$ is nef. Since $(X',B'+tf'^*D)$ is log smooth and since it is sub-lc but not sub-klt 
over the generic point of $D$, 
there is a prime divisor $S$ on $X'$ mapping onto $D$ such that $\mu_SB'+t\mu_Sf'^*D=1$.

Let $H\in |A|$ be a general member  and let $H',G,G'$ be its pullback to $Z',X,X'$, 
respectively. Since the centre of $D$ on $Z$ is positive-dimensional, $H'$ intersects $D$ and 
$G'$ intersects $S$. By divisorial generalised adjunction we can write 
$$
K_G+B_G+M_G\sim_\R (K_X+B+G+M)|_G
$$
where $(G,B_G+M_G)$ is generalised $\varepsilon$-lc with nef part $M_{G'}=M'|_{G'}$. 
Moreover, $-K_G$ is big over $H$, and 
$$
K_G+B_G+M_G\sim_\R g^*(L+A)|_H
$$
where $g$ denotes $G\to H$. Thus $(G,B_G+M_G)\to H$ is a generalised $(d,2^{d-1}r,\varepsilon)$-Fano type 
fibration. We can write 
$$
K_{G'}+B_{G'}+M_{G'}\sim_\R (K_{X'}+B'+G'+M')|_{G'}
$$
where $B_{G'}=B'|_{G'}$ and $K_{G'}+B_{G'}+M_{G'}$ is the pullback of $K_G+B_G+M_G$. 

Let $C$ be a component of $D\cap H'$ 
and let $s$ be the generalised lc threshold of $g'^*C$ with respect to $(G',B_{G'}+M_{G'})$ over the 
generic point of $C$ where $g'$ denotes $G'\to H'$. Then $1-s$ is the coefficient of $C$ in the 
discriminant b-divisor of adjunction for $(G,B_G+M_G)\to H$. Thus applying 
Theorem \ref{t-sh-conj-bnd-base-gen-fib} in dimension $d-1$, 
we deduce that $1-s\le 1-\delta$ for some $\delta>0$ depending only on  $d,r,\varepsilon$. Thus $s\ge \delta$.

By definition of $s$, 
for any prime divisor $T$ on $G'$ mapping onto $C$, 
we have  the inequality $\mu_TB_{G'}+s\mu_Tg'^*C\le 1$. In particular, 
if we take $T$ to be a component of $S\cap G'$ which maps onto $C$, then we have
$$
 \begin{array}{l l}
\mu_SB'+s\mu_Sf'^*D &=\mu_TB'|_{G'}+s \mu_Tf'^*D|_{G'}\\ 
&=\mu_TB_{G'}+s\mu_Tg'^*D|_{H'}\\ 
&=\mu_TB_{G'}+s\mu_Tg'^*C\\ 
&\le 1\\
&=\mu_SB'+t\mu_Sf'^*D
\end{array}
$$
 where we use the fact that over the generic point of $C$ 
the two divisors $g'^*D|_{H'}$ and $g'^*C$ coincide. Therefore, $\delta\le s\le t$, hence 
$\mu_DB_{Z'}=1-t\le 1-\delta$  where $B_{Z'}$ 
is the discriminant divisor on $Z'$ defined for $(X,B+M)$ over $Z$. Thus we have settled the case 
when the centre of $D$ on $Z$ is positive-dimensional.
 
From now on we can assume that the centre of $D$ on $Z$ is a closed point, say $z$. 
Let $Z',X', f', B',M'$ be as before. Pick $N\in |A|$ passing through $z$. 
Then
$$
f^*A+B+M-f^*N\sim B+M
$$ 
is obviously pseudo-effective. Thus by Theorem \ref{t-sing-gen-FT-fib-totalspace} in dimension $d$, 
the generalised lc threshold $u$ of 
$f^*N$ with respect to $(X,B+M)$ is bounded from below by some $\delta>0$ depending only on $d,r,\varepsilon$.

Since $N$ passes through $z$,  we have $\psi^*N\ge D$ where $\psi$ denotes $Z'\to Z$. Thus 
the generalised lc threshold $v$ of $f'^*\psi^*N$ with respect to 
$(X',B'+M')$ over the generic point of $D$ is at most as large as the generalised lc threshold $t$ of $f'^*D$ 
with respect to $(X',B'+M')$ over the generic point of $D$. On the other hand, 
the generalised lc threshold $u$ of $f^*N$ with respect to $(X,B+M)$ globally coincides with 
the generalised lc threshold of $f'^*\psi^*N$ with respect to 
$(X',B'+M')$ globally which is at most as large as the generalised lc threshold $v$ of $f'^*\psi^*N$ 
with respect to $(X',B'+M')$ over the generic point of $D$.
Therefore, $\delta\le u\le v\le t$, hence  $\mu_DB_{Z'}=1-t\le 1-\delta$.
\end{proof}



\vspace{2cm}

\textsc{Yau Mathematical Sciences Center,} \endgraf
\textsc{JingZhai Building, Tsinghua University} \endgraf
\textsc{ Hai Dian District, Beijing, China 100084  } \endgraf

\medskip

\email{birkar@tsinghua.edu.cn}

\end{document}